\documentclass[11pt]{article}

\usepackage{amssymb, amsthm, amsmath, amscd}
\usepackage{mathtools}
\usepackage{tikz-cd}

\newtheorem{theorem}{Theorem}[section]
\newtheorem{lemma}[theorem]{Lemma}
\newtheorem{proposition}[theorem]{Proposition}
\newtheorem{Notation}[theorem]{Notation}
\theoremstyle{definition}
\newtheorem{definition}[theorem]{Definition}
\newtheorem{corollary}[theorem]{Corollary}
\newtheorem{example}[theorem]{Example}

\newtheorem{remark}[theorem]{Remark}

\makeatletter
    
    \newcommand{\Rmnum}[1]{\expandafter\@slowromancap\romannumeral #1@}
  \makeatother

\newtheorem{thm}{Theorem}[section]

\newtheorem{cor}[thm]{Corollary}

\theoremstyle{definition}
\theoremstyle{definition}
\theoremstyle{definition}

\newcommand{\N}{\mathbb{N}}
\newcommand{\Z}{\mathbb{Z}}

\newcommand{\C}{\mathbb{C}}
\newcommand{\T}{\mathbb{T}}

\newcommand{\Hom}{\operatorname{Hom}}

\newcommand{\Aff}{\operatorname{Aff}}

\newcommand{\morp}{contractive completely positive linear map}

\newcommand{\ep}{\varepsilon}

\newcommand{\CA}{$C^*$-algebra}

\newcommand{\af}{{\alpha}}

\newcommand{\dist}{{\rm dist}}

\newcommand{\beq}{\begin{eqnarray}}
\newcommand{\eneq}{\end{eqnarray}}

\usepackage[all]{xy}

\title{Stability of rotation relations in $C^*$-algebras }
\author{Jiajie Hua and Qingyun Wang}
\date{}

\begin{document}

\maketitle

\begin{abstract}
Let $\Theta=(\theta_{j,k})_{3\times 3}$ be a non-degenerate real skew-symmetric $3\times 3$ matrix, where $\theta_{j,k}\in [0,1).$
For any $\varepsilon>0$, we prove that there exists $\delta>0$ satisfying the following: if $v_1,v_2,v_3$ are three unitaries in any unital simple separable $C^*$-algebra $A$ with tracial rank at most one, such that
$$\|v_kv_j-e^{2\pi i \theta_{j,k}}v_jv_k\|<\delta \,\,\,\, \mbox{and}\,\,\,\, \frac{1}{2\pi i}\tau(\log_{\theta}(v_kv_jv_k^*v_j^*))=\theta_{j,k}$$
for all $\tau\in T(A)$ and $j,k=1,2,3,$ where  $\log_{\theta}$ is a continuous branch of logarithm (see Definition \ref{logarithm}) for some real number $\theta\in [0, 1)$, then there exists a triple of unitaries $\tilde{v}_1,\tilde{v}_2,\tilde{v}_3\in A$ such that
$$\tilde{v}_k\tilde{v}_j=e^{2\pi i\theta_{j,k} }\tilde{v}_j\tilde{v}_k\,\,\,\,\mbox{and}\,\,\,\,\|\tilde{v}_j-v_j\|<\varepsilon,\,\,j,k=1,2,3.$$

The same conclusion holds if $\Theta$ is rational or non-degenerate and $A$ is a nuclear purely infinite
simple $C^*$-algebra (where the trace condition is vacuous). 

If $\Theta$ is degenerate and $A$ has tracial rank at most one or is nuclear purely infinite simple,
we provide some additional injectivity condition to get the above conclusion.
\end{abstract}

\section{Introduction}
An old and famous question in matrix and operator theory asks whether any pair of almost commuting self-adjoint matrices are norm close to a pair of exactly commuting self-adjoint matrices (\cite{Berg},\cite{Davidson},\cite{Halmos},\cite{Rosenthal}). The more accurate statement of this question is:

Let  $\varepsilon>0$, is there a $\delta>0$ such that if $A$ and $B$ are two $n\times n$ self-adjoint  matrices satisfying 
$$\|AB-BA\|<\delta,$$ 
then there exists a pair of self-adjoint matrices $\widetilde{A}$ and $\widetilde{B}$ in $M_n$ such that
$$\widetilde{A}\widetilde{B}=\widetilde{B}\widetilde{A},\,\,\,\,\|A-\widetilde{A}\|<\epsilon\,\,{\rm and}\,\,\|B-\widetilde{B}\|<\epsilon?$$
In the above question, it is important that $\delta$ is a universal constant independent of  matrix size $n.$
This question was solved affirmatively by Lin in the 1990's (See \cite{Halmos},\cite{Lin2}).
The corresponding question for a pair of unitary matrices is false, as pointed out by Voiculescu in \cite{Voiculescu1} and \cite{Voiculescu}. However the story does not end here. An obstruction has been found by Exel and Loring in \cite{EL}. The answer becomes yes if this obstruction vanishes. See also  (\cite{EL},\cite{ELP1},\cite{ELP},\cite{ExL1},\cite{ExL2}).

A natural generalization of the second question above is to see what happens for pairs of unitaries that almost commute up to a scalar with norm 1. It turns out that similar conclusion holds, and in fact one can deal with more general ambient $C^*$-algebras rather than just matrix algebras. More precisely, in \cite{Hua-Lin}, the first author and H. Lin proved the following:

\begin{theorem} \label{HH} Let $\theta$ be a real number in $(-\frac{1}{2},\frac{1}{2})$. For any $\varepsilon>0,$  there is a $\delta>0$, depending only on $\varepsilon$ and $\theta$, such that if $u$ and $v$ are two unitaries in any unital simple separable $C^*$-algebra $A$ with tracial rank zero satisfying

\beq \|uv-e^{2\pi i\theta}vu\|<\delta\quad{\rm and} \\
\frac{1}{2\pi i}\tau(\log(uvu^*v^*))=\theta\quad \label{tc}
\eneq
for all tracial state $\tau$ of $A$, then there exists a pair of unitaries $\tilde{u}$ and $\tilde{v}$ in $A$ such that
\[
\tilde{u}\tilde{v}=e^{2\pi i\theta}\tilde{v}\tilde{u},\,\,\,\,\|u-\tilde{u}\|<\varepsilon\, \,\,\,{\rm and}\,\,\,\,\|v-\tilde{v}\|<\varepsilon.
\]
\end{theorem}

Note that the trace condition (\ref{tc}) is also necessary, see Corollary \ref{C_ETF}.

Let $\theta\in\mathbb{R}$. We call a pair of unitaries $u, v$ with $ uv = e^{2\pi i \theta} vu$ to satisfy the \emph{rotation relation} with respect to $\theta$,
since the universal $C^*$-algebra generated by such unitaries is the rotation algebra.
So another way to phrase Thoerem \ref{HH} is to say that the
rotation relation is stable in unital simple separable $C^*$-algebras with tracial rank zero, providing that the trace condition (\ref{tc}) is satisfied.

One can then further ask the stability of the rotation relations
for a $n$-tupe of unitaries.
Let $\Theta=(\theta_{j,k})_{n \times n}$ be a real skew-symmetric $n$ by $n$ matrix.
We say unitaries $u_1, u_2, \dots u_n$ in a $C^*$-algebra $A$
satisfies the rotation relations with respect to $\Theta$ if
$u_ku_j=e^{2\pi i \theta_{j,k}}u_ju_k$,
for all $1\leq j, k \leq n$.

We mainly focus on the case $n = 3$ in this paper. In view of what happens for $n=2$, it is natural to ask the following question:\\
(\textbf{Q1\label{conj}}):  \emph{Let $\Theta=(\theta_{j,k})_{3\times 3}$ be a real skew-symmetric $3\times 3$ matrix, where $\theta_{j,k}\in (\frac{1}{2},-\frac{1}{2}).$ For any $\varepsilon>0$, is there a  $\delta>0$ satisfying the following: For any unital simple separable $C^*$-algebra $A$ with tracial rank zero
which satisfies the UCT, any three unitaries $v_1,v_2,v_3\in A$  such that
$$\|v_kv_j-e^{2\pi i \theta_{j,k}}v_jv_k\|<\delta \,\,\,\, \mbox{and}\,\,\,\, \frac{1}{2\pi i}\tau(\log(v_kv_jv_k^*v_j^*))=\theta_{j,k}$$
for all $\tau\in T(A)$ and $j,k=1,2,3,$ there exists a triple of unitaries $\tilde{v}_1,\tilde{v}_2,\tilde{v}_3\in A$ such that
$$\tilde{v}_k\tilde{v}_j=e^{2\pi i\theta_{j,k} }\tilde{v}_j\tilde{v}_k\,\,\,\,\mbox{and}\,\,\,\,\|\tilde{v}_j-v_j\|<\varepsilon,\,\,j,k=1,2,3?$$}


In this paper, we  confirm (Q\ref{conj}) when $\Theta$ is a non-degenerate real skew-symmetric $3\times 3$ matrix and extend the result to the class of $C^*$-algebras with tracial rank at most one.
However, when $\Theta$ is a degenerate real skew-symmetric $3\times 3$ matrix, (Q\ref{conj}) is not true. We  give a counter example in Section \ref{S_SofTR1}. We then provide a sufficient condition to make (Q\ref{conj}) set up. We also deal with the purely infinite simple case and obtain similar results.

This paper is organized as following. In Section 2, we list some notations and known results. 
In Section 3, we give a very concrete description of the $K$-theory of the rotation algebra.
In Section 4, we further generalize the Exel trace formula.  In Section 5, we prove stability of the rotation relations in the class 
of $C^*$-algebras of tracial rank at most one with an additional injective condition. We then show that this
condition is automatic if $\Theta$ is non-degenerate, thus confirm (Q\ref{conj}) in this case.
In the last section, we deal with the class of nuclear purely infinite simple $C^*$-algebras. We show that
stability of the rotation relations always holds if $\Theta$ is either rational or non-degenerate.

\section{Preliminaries}
Let $n\geq2$ be an integer and denote by $\mathcal{T}_n$ the space of $n\times n$ real skew-symmetric matrices.

\begin{definition}(See also section $1$ of \cite{Rieffel1}) Let $\Theta=(\theta_{j,k})_{n \times n}\in \mathcal{T}_n$. The noncommutative tori $\mathcal{A}_{\Theta}$ is the universal $C^*$-algebra generated by unitaries $u_1,u_2,\dots,u_n$ subject to the relations
$$u_ku_j=e^{2\pi i\theta_{j,k}}u_ju_k$$
for $1\leq j,k\leq n.$ (Of course, if all $\theta_{j,k}$ are integers, it is not really noncommutative.)
\end{definition}

For any real skew symmetric matrix $\Theta=(\theta_{j,k})_{n \times n},$   $\mathcal{A}_{\Theta}$ has a canonical tracial state $\tau_{\Theta}$ given by the integration over
the canonical action of $\widehat{\mathbb{Z}^n}$ (See page 4 of \cite{Rieffel1} for more details). 
We denote this trace by $\tau_{\Theta}$.

\begin{definition}
A skew symmetric real $n \times n$ matrix $\Theta$ is {\it nondegenerate} if whenever $x\in \mathbb{Z}^n$ satisfies $\exp(2\pi i\langle x,\Theta y\rangle)=1$ for all $y\in \mathbb{Z}^n,$ then $x=0.$ Otherwise, we say $\Theta$ is {\it degenerate}.

We say $\Theta$ is {\it rational} if its entries are all rational numbers; otherwise it is said to be {\it nonrational}.
\end{definition}


The following result is well known.
\begin{theorem}[Theorem 1.9 of \cite{Phillips-06}]\label{simple and tracial} The $C^*$-algebra $\mathcal{A}_{\Theta}$ is simple if and only if $\Theta$ is nondegenerate. Moreover, if $\mathcal{A}_{\Theta}$ is simple it has a unique tracial state $\tau_{\Theta}$.
\end{theorem}

\begin{lemma}[Lemma 3.1 of \cite{Itz-Phillips}]\label{3-dimensional rank}
Let $\Theta\in \mathcal{T}_3$, and
$$\Theta=\left(
\begin{array}{ccc}
0 & \theta_{1,2} & \theta_{1,3} \\
-\theta_{1,2} & 0 & \theta_{2,3} \\
-\theta_{1,3} & -\theta_{2,3} & 0 \\
\end{array}
\right).$$
Then $\Theta$ is nondegenerate if and only if $\dim_{\mathbb{Q}}(\rm{span}_{\mathbb{Q}}(1,\theta_{1,2},\theta_{1,3},\theta_{2,3}))\geq 3.$
\end{lemma}

\begin{Notation}
{\rm
Let $A$ be a unital \CA. Denote by $T(A)$ the tracial state space of $A.$ The set of all faithful tracial
states on $A$ will be denoted by $T_f(A)$. Denote by ${\rm Aff}(T(A))$ the space of all real affine continuous functions on $T(A).$ If $\tau\in T(A)$, we will use $\tau^{\oplus k}$ for the trace $\tau\otimes {\rm Tr}$ on $M_k (A)$ for all integer $k \geq1$, where Tr is the
unnormalized trace on the matrix algebra $M_k$. Denote by $A_{s.a.}$ the set of all self-adjoint elements in $A$. If $a\in A_{s.a.}$, denote by $\breve{a}$ the real affine function in ${\rm Aff}(T(A))$ defined by $\breve{a}(\tau)=\tau(a)$ for all
$\tau\in T(A).$

Denote by $U_n(A)$ the group of unitaries in $M_n(A)$ for $n\geq 1$. We often use $U(A)$ to express $U_1(A)$. Denote by
$U_0(A)$ the subgroup of $U(A)$ consists of unitaries path connected to $1_A$. Denote by $CU(A)$ the 
closure of the subgroup generated by commutators of $U(A)$. 
Let $U_{\infty}(A)$ be the increasing union of $U_n(A), n=1, 2, \dots$,  viewed as a topological group
with the inductive limit topology. Define $U_{\infty, 0}(A)$ and $CU_{\infty}(A)$ in a similar fashion.

Let $A$ be a unital \CA\, and let $u\in A$ be a unitary.
Define ${\rm Ad}\, u(a)=u^*au$ for all $a\in A.$ 
For any $a\in A,$ denote by $\mathrm{sp}(a)$ the spectrum of $a.$

Denote by $\rho_A: K_0(A)\to \Aff(T(A))$ the order preserving map
induced by $\rho_{A}([p])(\tau)=\tau^{\oplus n} (p)$ for all projections $p\in A\otimes M_n,$ $n=1,2,....$
}
\end{Notation}

\begin{definition}[ \cite{Rieffel3}] \label{def of Rieffel P}Let $\theta\in (0,1).$ Choose $\epsilon$ such that $0<\epsilon \leq \theta<\theta+\epsilon\leq 1.$ Set 
$$f(e^{2\pi it})=\left\{
\begin{aligned}
\epsilon^{-1}t,\phantom{ttttttttttt}& & 0\leq t\leq \epsilon,\phantom{tttt}\\
1,\phantom{tttttttttttttt}& & \epsilon\leq t\leq \theta,\phantom{tttt}\\
\epsilon^{-1}(\theta+\epsilon-t),& & \theta\leq t\leq \theta+\epsilon,\\
0,\phantom{tttttttttttttt} & & \theta+\epsilon \leq t\leq 1,\\
\end{aligned}\right.
$$
and
$$g(e^{2\pi it})=\left\{
\begin{aligned}
0,\phantom{tttttttttttttttttttttttttttt}& & 0\leq t\leq \theta,\phantom{tttt}\\
[f(e^{2\pi it})(1-f(e^{2\pi it}))]^{1\slash 2},& & \theta\leq t\leq \theta+\epsilon,\\
0,\phantom{tttttttttttttttttttttttttttt}& & \theta+\epsilon\leq t\leq 1.\\
\end{aligned}\right.
$$
Then $f$ and $g$ are the real-valued functions on the circle which satisfy 

(1) $g(e^{2\pi it})\cdot g(e^{2\pi i(t-\theta)})=0, $

(2) $g(e^{2\pi it})\cdot [f(e^{2\pi it})+f(e^{2\pi i(t+\theta)})]=g(e^{2\pi it})$ and

(3) $f(e^{2\pi it})=[f(e^{2\pi it})]^2+[g(e^{2\pi it})]^2 +[g(e^{2\pi i(t-\theta)})]^2.$

Let $\mathcal{A}_{\theta}$ be the universal $C^*$-algebra generated by unitaries $u,v$ subject to the relations
$uv=e^{2\pi i\theta}vu.$  The {\it Rieffel projection} in $\mathcal{A}_\theta$ is the projection $$p = g(u)v^* + f(u) + vg(u)$$ (see the construction in Theorem 1.1 of \cite{Rieffel3}).
\end{definition}

\begin{definition}\label{D_totalK}
 Let $A$ be a $C^*$-algebra. Following Dadarlat and Loring (\cite{DL}), we set 
 $$\underline{K}(A) = 
\oplus_{n=1}^{\infty} \left(K_0(A; \Z/n\Z) \oplus K_1(A; \Z/n\Z)\right).
$$
Let $B$ be another unital $C^*$-algebra. If furthermore, $A$ is assumed to be separable and satisfy
the Universal Coefficient Theorem (\cite{Rosenberg-Schochet-Duke}), by \cite{DL},
$$KL(A, B) \cong Hom_{\Lambda}(\underline{K}(A), \underline{K}(B)).$$
Here $KL(A, B) = KK(A, B)/Pext(K_*(A), K_*(B))$, where $Pext$ is the subgroup of $Ext_{\Z}^{1}$
consists of classes of pure extensions (see \cite{DL} for details).
\end{definition}

We recall the definition of tracial rank of
$C^*$-algebras:
\begin{definition}\label{Dtr1} (\cite{LinPLMS}) Let $A$ be a unital simple $C^*$-algebra.
We say $A$ has {\it tracial rank at most one},  if for any $\varepsilon>0$, any
finite subset $\mathcal{F}\subset A$  and any
nonzero positive element $c\in A,$ there exists a 
$C^*$-subalgebra $B \subset A$ with $1_{B}=p\neq 0$ such
that:\\
(\romannumeral1) $B \cong \oplus_{j=1}^{n} P_jM_{k_j}(C(X_j))P_j$, where $X_j$ is a compact metrizable space
with covering dimension $\leq 1$, $P_j$ is a projection in $M_{k_j}(C(X_j))$, for $j = 1, 2, \dots, n$.\\
(\romannumeral2) $\|pa-ap\|<\varepsilon$ for all $a\in \mathcal{F}$,\\
(\romannumeral3) $\dist (pap, B)<\varepsilon$ for all $a\in \mathcal{F}$  and\\
(\romannumeral4) $1_A-p$ is Murray-von Neumann equivalent to a projection in $\overline{cAc}.$
\end{definition}

If $A$ has tracial rank at most one, we write $\mathrm{TR}(A) \leq 1.$

\begin{definition} Let $L : A \rightarrow B$ be a linear map. Let $\delta> 0$ and $\mathcal{G}\subset A$ be a finite subset. We say $L$ is
$\mathcal{G}$-$\delta$-multiplicative if
$$\|L(ab)-L(a)L(b)\|<\delta\,\,\, {\rm for\,\,\, all}\,\,\, a, b \in \mathcal{G}.$$
\end{definition}

Let $C_n$ be a $C^*$-algebra such that $K_*(A; \Z/n\Z) \cong K_*(A \otimes C_n)$.
For convenience, if $L \colon A \rightarrow B$ is a linear map, we will use the same symbol $L$
to denote the induced map $L \otimes \mathrm{id}_{n} \colon A \otimes M_n \rightarrow B \otimes M_n$.
as well as $L \otimes \mathrm{id}_{C_n} \colon A \otimes C_n \rightarrow B \otimes C_n$.

It is well known that if $a \in M_n(A)$ is an `almost' projection, then it is norm close to a projection. Two norm close projection are unitarily equivalent.
So $[a] \in K_0(A)$ is well-defined. Similarly, if $b \in M_n(A)$ is an 
`almost' unitary, we shall use $[b]$ to denote the equivalent class in 
$K_1(A)$. If $L \colon A \rightarrow B$ is an `almost' homomorphism, we shall
use $[L]$ to denote the induced (partially defined) map on the $K$-theories.
From \cite{Dadarlat-IJM} or \cite{Lin-book}, we can know that for any finite set $\mathcal{P} \subset \underline{K}(A)$, there is a finite subset $\mathcal{G} \subset A$ and $\eta > 0$ such that,
for any unital completely positive $\mathcal{G}$-$\delta$-multiplicative linear map $L$, $[L]$ is well defined on $\mathcal{P}$. 

\begin{Notation}
If $u$ is a unitary in $U_{\infty}(A)$, we shall use $[u]$ to denote the equivalence class
in $K_1(A)$ and use $\bar{u}$ to denote the equivalent class in $U_{\infty}(A)/CU_{\infty}(A)$.
If $L \colon A \rightarrow B$ is an `almost' homomorphism
so that $L(u)$ is invertible, we shall define 
\[
\langle L(u) \rangle = L(u)(L(u)^*L(u))^{-1/2}.
\]
\end{Notation}

\section{K-theory of rotation algebra}
Any noncommutative tori can be written as an iterated crossed product by $\Z$. It then follows quickly from the Pimsner-Voiculsecu six-term
exact sequence that $K_*(\mathcal{A}_\Theta) \cong \Z^{2^{n}-1}$, for any
$\Theta \in \mathcal{T}_n$. For our purpose, we need a more specific
basis for the K-theory.
We start with $ n=2 $. Theorem 3.1 and Proposition 3.3 below are surely well known.

\begin{theorem} \label{T_KofRA}
Let $\theta \in [0, 1)$. Let $z$ be the identity
function of $C(\T)$. Let $\alpha \colon \mathbb{Z} \curvearrowright C(\T)$ be the action
determined by $\alpha(z) = e^{2\pi i \theta} z$. Then $\mathcal{A}_\theta$ is
naturally isomorphic to the crossed product $C(\T) \rtimes _\alpha \Z$.
Moreover, there is a short exact sequence induced from the
Pimsner-Voiculsecu six-term exact sequence from this crossed product:

\[
 0\longrightarrow  K_0(C(\T))\stackrel{\phantom{aaa}i_{*0}\phantom{aa}}{\longrightarrow}K_0(\mathcal{A}_{\theta})\stackrel{\phantom{aaa}\partial \phantom{aa}}{\longrightarrow} K_{1}(C(\T))\longrightarrow 0.
\]
\end{theorem}

\begin{definition}\label{D_RieffelProj}
Let $\theta \in [0, 1)$. Let $u$, $v$ be the canonical generators of $\mathcal{A}_{\theta}$. If $\theta \neq 0$, we define $b_{u, v} \in K_0(\mathcal{A}_{\theta})$ to be the equivalent class of the Rieffel projection as constructed in the proof of Theorem 1.1 of \cite{Rieffel3}. If $\theta = 0$, we let $b_{u, v} \in K_0(\mathcal{A}_{\theta})$ be the bott element. (See Definition 2.7 of \cite{Hua-Lin}.)\\
\end{definition}

\begin{proposition}\label{P_RieffelProj}
Let $\theta \in [0, 1)$. Let $\tau$ be the canonical tracial state on $\mathcal{A}_\theta$.
Let $b_{u, v} \in K_0(\mathcal{A}_{\theta})$ be defined as in Definition \ref{D_RieffelProj}.
Then $\tau(b_{u, v}) = \theta$.
Moreover,
if $\partial \colon K_0(\mathcal{A}_{\theta}) \rightarrow K_1(C(\T))$
is the homomorphism defined as in Theorem \ref{T_KofRA},
then $\partial (b_{u, v}) = [z]$,
where $z$ is the identity function of $C(\T)$.
\end{proposition}

For $n = 3$, we have the following result:

\begin{proposition}\label{P_KofNCT}
Let $\Theta \in \mathcal{T}_3$. Let $u_1, u_2, u_3$ be the canonical
generators of $\mathcal{A}_\Theta$. For $k=1,2,3$ and $j=1,2,3$, let $b_{u_k, u_j}$ be the the $K_0$-element of $\mathcal{A}_\Theta$
as defined in Definition \ref{D_RieffelProj}
(viewing $\mathcal{A}_{\theta_{j, k}}$ as a subalgebra of $\mathcal{A}_\Theta$).
Then
\begin{enumerate}
    \item $K_0(\mathcal{A}_\Theta)$ is generated by
    $[1_{\mathcal{A}_{\Theta}}], b_{u_2,u_1}, b_{u_3,u_1}, b_{u_3,u_2}$.
    \item $K_1(\mathcal{A}_\Theta)$ is generated by $[u_1], [u_2], [u_3]$
    and a fourth element $[u]$.
\end{enumerate}
\end{proposition}

\begin{proof}
Let $\mathcal{A}_{\theta_{1,2}}$ denote the universal $C^*$-algebra
generated by a pair of unitaries $u_1$ and $u_2$
subject to $u_2u_1=e^{2\pi i\theta_{1,2}}u_1 u_2.$
Define an action $\alpha$ on $\mathcal{A}_{\theta_{1,2}}$ by
$\alpha(u_1)=u_3u_1u_3^*=e^{2\pi i \theta_{1,3}}u_1$
and $\alpha(u_2)=u_3u_2u_3^*=e^{2\pi i \theta_{2,3}}u_2$.
Then $\mathcal{A}_{\Theta}=\mathcal{A}_{\theta_{1,2}}\rtimes_{\alpha}\mathbb{Z}.$ We have the following Pimsner-Voiculescu
six-term exact sequence:

$$\xymatrix{
  K_0(\mathcal{A}_{\theta_{1,2}})  \ar[r]^{\rm{id}_{*0}-\alpha_{*0}} &  K_0(\mathcal{A}_{\theta_{1,2}}) \ar[r]^{\phantom{aaa}\iota_{*0}\phantom{aa}} &  \ar[d]^{\partial}K_0(\mathcal{A}_{\Theta}) \\
  K_1(\mathcal{A}_{\Theta}) \ar[u]^{\delta} & K_{1}(\mathcal{A}_{\theta_{1,2}}) \ar[l]_{\phantom{aa}\iota_{*1}\phantom{aa}} & K_{1}(\mathcal{A}_{\theta_{1,2}}) \ar[l]_{\rm{id}_{*1}-\alpha_{*1}}.}
$$
Since $\alpha$ is homotopic to the identity, we get two short exact sequences
\begin{equation}
 0\longrightarrow  K_0(\mathcal{A}_{\theta_{1,2}})\stackrel{\phantom{aaa}\iota_{*0}\phantom{aa}}{\longrightarrow}K_0(\mathcal{A}_{\Theta})\stackrel{\phantom{aaa}\partial \phantom{aa}}{\longrightarrow} K_{1}(\mathcal{A}_{\theta_{1,2}})\longrightarrow 0
 \label{esK0}
\end{equation}
and
\begin{equation}
 0\longrightarrow K_1(\mathcal{A}_{\theta_{1,2}}) \stackrel{\phantom{aaa}\iota_{*1}\phantom{aa}}{\longrightarrow}K_1(\mathcal{A}_{\Theta})\stackrel{\phantom{aaa}\delta\phantom{aa}}\longrightarrow K_{0}(\mathcal{A}_{\theta_{1,2}})\longrightarrow 0.
 \label{esK1}
\end{equation}

By Lemma 1.2 of \cite{PimV}, $K_1(\mathcal{A}_\Theta)$ is generated by the classes of unitaries of the form
\[
(1_{\mathcal{A}_{\theta_{1,2}}} \otimes 1_n - P) + Px(u_3^*\otimes 1_n)P, \text{\,\,where\,\,} P, x \in \mathcal{A}_{\theta_{1,2}}\otimes M_n\,\,{\rm and}\,\,
P \text{\,\,is a projection}.
\]
From the proof of Lemma 2.3 of \cite{PimV},
the connecting map $\delta \colon K_1(\mathcal{A}_\theta) \rightarrow K_0(A_{\theta_{1,2}})$ is determined by:
\[
\delta [(1_{\mathcal{A}_{\theta_{1,2}}} \otimes 1_n - P) + Px(u_3^*\otimes 1_n)P] = [P].
\]
In particular, we have $\delta ([u_3]) = [1_{\mathcal{A}_{\theta_{1,2}}}]$. Let $[u]$ be an element in $K_1(\mathcal{A}_{\theta_{1,2}})$ such that
$\partial [u] = b_{u_2, u_1}$.
Note that $[1_{\mathcal{A}_{\theta_{1,2}}}]$ and
$b_{u_2, u_1}$ are generators of $K_0(\mathcal{A}_{\theta_{1,2}})$
by Proposition \ref{P_RieffelProj}.
By Corollary 2.5 of \cite{PimV}, $K_1(\mathcal{A}_{\theta_{1,2}})$ is isomorphic to $\mathbb{Z}^2$
which is generated by $[u_1]$ and $[u_2]$.
From the exact sequence (\ref{esK1}) we see that
$K_1(\mathcal{A}_{\theta_{1,2}})$ is generated by
$[u_1], [u_2], [u_3]$ and $[u]$.

By the same argument as in the proposition
in the appendix of \cite{PimV}, we see that
the map $\partial \colon K_0(\mathcal{A}_{\Theta}) \rightarrow K_1(\mathcal{A}_{\theta_{1,2}})$
is determined by $\partial (b_{u_3,u_1})=[u_1]$,  $\partial (b_{u_3,u_2})=[u_2]$, then
the exactness of the sequence (\ref{esK0})
shows that $K_0(\mathcal{A}_\Theta)$ is generated by
$[1_{\mathcal{A}_{\Theta}}], b_{u_2,u_1}, b_{u_3,u_1}, b_{u_3,u_2}$.
\end{proof}

\section{Generalized Exel trace formula}
The Exel trace formula is one of the main ingredient in the proof of our main results. We need some preparation
before we can state the Exel trace formula (in our generalized form).

\begin{definition} (see \cite{Farsi}) For $0\leq \varepsilon\leq 2$ and $\theta\in [0, 1),$ the soft rotation algebras $\mathcal{S}_{\varepsilon,\theta}$ is defined to be
the universal $C^*$-algebra generated by a pair of unitaries
$\mathfrak{u}_{\varepsilon,\theta}$ and $\mathfrak{v}_{\varepsilon,\theta}$ subject to $\|\mathfrak{u}_{\varepsilon,\theta}\mathfrak{v}_{\varepsilon,\theta}-e^{2\pi i\theta}\mathfrak{v}_{\varepsilon,\theta}\mathfrak{u}_{\varepsilon,\theta}\|\leq \varepsilon$.

In particular, we have $\mathcal{S}_{0,\theta}=\mathcal{A}_{\theta}.$
\end{definition}

\begin{definition} For $0\leq \varepsilon\leq 2$ and $\theta\in [0, 1),$ let $B_{\varepsilon,\theta}$ be the universal $C^*$-algebra generated by unitaries $x_n,n\in \mathbb{Z}$, subject to the relations $\|x_{n+1}-e^{2\pi i\theta}x_n\|\leq \varepsilon.$ Let $\alpha_{\varepsilon,\theta}$ be the automorphism of $B_{\varepsilon,\theta}$ specified by $\alpha_{\varepsilon,\theta}(x_n)=x_{n+1}.$
\end{definition}

\begin{proposition}[Proposition 2.2 of \cite{Farsi}]\label{soft isomorphism} For $0<\varepsilon<2,$
$B_{\varepsilon,\theta}\rtimes_{\alpha_{\varepsilon,\theta}}\mathbb{Z}$ is isomorphic to $\mathcal{S}_{\varepsilon,\theta}$.
\end{proposition}

\begin{theorem}[Theorem 2.3 of \cite{Farsi}, Theorem 2.4 of \cite{Exel}]\label{soft k-theory}
Assume $0 \leq \varepsilon <2 $.
Let $z$ denote the canonical generator of $C^*$-algebra $C(\mathbb{T}).$ Identify $\mathcal{A}_\theta$ as the crossed product of $C(\T)$ by the action $\alpha$ of $\Z$ induced by
$\alpha (z) = e^{2\pi i \theta} z$. We have the following:

(1) Let $\psi_{\varepsilon}^{\theta}:B_{\varepsilon,\theta}\rightarrow C(\mathbb{T})$ be the unique homomorphism such that $\psi_{\varepsilon}^{\theta}(x_n)=e^{2n\pi i\theta}z$ for all $n\in \mathbb{Z}$.
Then $\psi_{\varepsilon}^{\theta}$ induces a homotopy equivalence between $B_{\varepsilon,\theta}$ and $C(\mathbb{T})$.

(2) Let $\varphi_{\varepsilon}^{\theta}$ be the homomorphism defined by
\[
\varphi_{\varepsilon}^{\theta}:\mathcal{S}_{\varepsilon,\theta}\rightarrow\mathcal{S}_{0,\theta}=\mathcal{A}_{\theta}, \quad\quad
\varphi_{\varepsilon}^{\theta}(\mathfrak{u}_{\varepsilon,\theta})=\mathfrak{u}_{0, \theta},\quad
\varphi_{\varepsilon}^{\theta}(\mathfrak{v}_{\varepsilon,\theta})=\mathfrak{v}_{0, \theta}.
\]
Then we have the following commutative diagram:

$$\xymatrix{
  0  \ar[r]^{} & K_0(B_{\ep,\theta}) \ar[d]_{\psi_{\ep_*}^{\theta}} \ar[r]^{} & K_0(\mathcal{S}_{\ep,\theta}) \ar[d]_{\varphi_{\ep_*}^{\theta}} \ar[r]^{\partial} & K_{1}(B_{\ep,\theta}) \ar[d]_{\psi_{\ep_*}^{\theta}} \ar[r]^{} & 0  \phantom{,}\\
  0 \ar[r]^{} & K_{0}(C(\mathbb{T})) \ar[r]^{} & K_0(\mathcal{A}_{\theta}) \ar[r]^{\partial} & K_{1}(C(\mathbb{T})) \ar[r]^{} & 0, }$$
where all vertical maps are isomorphisms
and all rows are derived from the Pimsner-Voiclescu exact sequences.
\end{theorem}

\begin{definition}\label{dbot}
Let $\theta \in [0, 1)$. Let $b = b_{u, v} \in K_0(\mathcal{A}_\theta)$ be the class
of the Rieffel projection if $\theta \neq 0$ or the bott element if $\theta = 0$
(See Definition \ref{D_RieffelProj}).
We then define
$b_{\ep}^{\theta}$ the element in $K_0(\mathcal{S}_{\ep,\theta})$
to be the element
$b_{\ep}^{\theta}=(\varphi_{\ep_{*}}^{\theta})^{-1}(b),$
where $\varphi_{\varepsilon}^{\theta}$ is defined as in Theorem \ref{soft k-theory}.
\end{definition}

It follows immediately from the definition that $\partial(b_{\ep}^{\theta})=[x_0]$ in $K_{1}(B_{\ep,\theta}).$

\begin{definition}\label{dbotA}
Let $A$ be a unital $C^*$-algebra and let $u$ and $v$ be two unitaries in $A$ such that $\|uv-e^{2\pi i\theta}vu\|\leq \varepsilon<2$. There is a homomorphism $\phi_{u,v}^{\theta}:\mathcal{S}_{\varepsilon,\theta}\rightarrow A$ such that $\phi_{u,v}^{\theta}(\mathfrak{u}_{\varepsilon,\theta})=u$ and $\phi_{u,v}^{\theta}(\mathfrak{v}_{\varepsilon,\theta})=v$.
We define $b^{\theta}_{u,v}=(\phi_{u,v}^{\theta})_{*0}(b_{\varepsilon}^{\theta})$.  Note that $b^{\theta}_{u,v}$ does not depend on $\varepsilon$ as long as $2>\varepsilon >\|uv-e^{2\pi i\theta}vu\|$.
\end{definition}

\begin{proposition}Let $\theta_1,\theta_2\in [0,1)$, let $A$ be a unital $C^*$-algebra and let $u$ and $v$ be two unitaries in $A$. If $\|uv-e^{2\pi i\theta_1}vu\|<\frac{1}{2}$ and
 $\|uv-e^{2\pi i\theta_2}vu\|<\frac{1}{2}$. Then $b^{\theta_1}_{u,v}=b^{\theta_2}_{u,v}$ in $K_0(A).$
\end{proposition}
\begin{proof}  By $\|uv-e^{2\pi i\theta_1}vu\|<\frac{1}{2}$ and
 $\|uv-e^{2\pi i\theta_2}vu\|<\frac{1}{2}$, we have $$|e^{2\pi i\theta_1}-e^{2\pi i\theta_2}|\leq \|uvu^*v^*-e^{2\pi i\theta_1}\|+\|uvu^*v^*-e^{2\pi i\theta_2}\|<1.$$
 So we can find $\theta\in \mathbb{R}$ such that $|e^{2\pi i\theta_1}-e^{2\pi i\theta}|=|e^{2\pi i\theta_2}-e^{2\pi i\theta}|<1.$
 Let $\|uv-e^{2\pi i\theta_1}vu\|=\varepsilon_1$ and $\|uv-e^{2\pi i\theta_2}vu\|=\varepsilon_2$.
We have  $$\|\mathfrak{u}_{\varepsilon_1,\theta_1}\mathfrak{v}_{\varepsilon_1,\theta_1}-e^{2\pi i\theta}\mathfrak{v}_{\varepsilon_1,\theta_1}\mathfrak{u}_{\varepsilon_1,\theta_1}\|\leq \|\mathfrak{u}_{\varepsilon_1,\theta_1}\mathfrak{v}_{\varepsilon_1,\theta_1}-e^{2\pi i\theta_1}\mathfrak{v}_{\varepsilon_1,\theta_1}\mathfrak{u}_{\varepsilon_1,\theta_1}\|+|e^{2\pi i\theta_1}-e^{2\pi i\theta}|\leq \varepsilon_1+1<\frac{3}{2}.$$
 Similarly,
 $$\|\mathfrak{u}_{\varepsilon_2,\theta_2}\mathfrak{v}_{\varepsilon_2,\theta_2}-e^{2\pi i\theta}\mathfrak{v}_{\varepsilon_2,\theta_2}\mathfrak{u}_{\varepsilon_2,\theta_2}\|<\frac{3}{2}.$$
 Thus we have a surjective homomorphism $\phi_{\mathfrak{u}_{\varepsilon_1,\theta_1},\mathfrak{v}_{\varepsilon_1,\theta_1}}$ from $\mathcal{S}_{\frac{3}{2},\theta}$ to
 $\mathcal{S}_{\varepsilon_1,\theta_1}$ and a surjective homomorphism $\phi_{\mathfrak{u}_{\varepsilon_2,\theta_2},\mathfrak{v}_{\varepsilon_2,\theta_2}}$ from $\mathcal{S}_{\frac{3}{2},\theta}$ to
 $\mathcal{S}_{\varepsilon_2,\theta_2}$ such that
 $$\phi^{\theta}_{\mathfrak{u}_{\varepsilon_1,\theta_1},\mathfrak{v}_{\varepsilon_1,\theta_1}}(\mathfrak{u}_{\frac{3}{2},\theta})=\mathfrak{u}_{\varepsilon_1,\theta_1},\,\,\,\,\,
 \phi^{\theta}_{\mathfrak{u}_{\varepsilon_1,\theta_1},\mathfrak{v}_{\varepsilon_1,\theta_1}}(\mathfrak{v}_{\frac{3}{2},\theta})=\mathfrak{v}_{\varepsilon_1,\theta_1}$$
 and
 $$\phi^{\theta}_{\mathfrak{u}_{\varepsilon_2,\theta_2},\mathfrak{v}_{\varepsilon_2,\theta_2}}(\mathfrak{u}_{\frac{3}{2},\theta})=\mathfrak{u}_{\varepsilon_2,\theta_2},\,\,\,\,\,
 \phi^{\theta}_{\mathfrak{u}_{\varepsilon_2,\theta_2},\mathfrak{v}_{\varepsilon_2,\theta_2}}(\mathfrak{v}_{\frac{3}{2},\theta})=\mathfrak{v}_{\varepsilon_2,\theta_2}.$$

So \beq
b^{\theta_1}_{u,v}&=&(\phi_{u,v}^{\theta_1})_*(b_{\varepsilon_1}^{\theta_1})=(\phi_{u,v}^{\theta_1})_*\circ(\phi^{\theta}_{\mathfrak{u}_{\varepsilon_1,\theta_1},\mathfrak{v}_{\varepsilon_1,\theta_1}}(\mathfrak{u}_{\frac{3}{2},\theta}))_*(b_{\frac{3}{2}}^{\theta})
\nonumber\\&=&(\phi_{u,v}^{\theta})_*(b_{\frac{3}{2}}^{\theta})\nonumber\\
&=&(\phi_{u,v}^{\theta_2})_*\circ(\phi^{\theta}_{\mathfrak{u}_{\varepsilon_2,\theta_2},\mathfrak{v}_{\varepsilon_2,\theta_2}}(\mathfrak{u}_{\frac{3}{2},\theta}))_*(b_{\frac{3}{2}}^{\theta})\nonumber\\
&=&(\phi_{u,v}^{\theta_2})_*(b_{\varepsilon_2}^{\theta_2})\nonumber\\
&=&b^{\theta_2}_{u,v}.\nonumber
\eneq
\end{proof}

In case $\theta = 0$, there is a more concrete construction of the $K_0$ element $b^{0}_{u,v}$ which is called the bott element of $u$ and $v$.
The interested reader are referred to \cite{Lor}, \cite{ExL1}, \cite{ExL2} and \cite{Hua-Lin} for the definition and more details about the bott element.
For $\theta \neq 0$, we shall give a similar construction for $b^{\theta}_{u, v}$.

\begin{definition}
Let $\theta \in (0, 1)$. Let $A$ be a unital $C^*$-algebra. Let $u, v$ be a pair of unitaries in $A$. We define $e^{\theta}(u, v)$ to
be the element $g(u)v^* + f(u) + vg(u)$, where $f, g$ are the functions on the circle
defined as in Definition \ref{def of Rieffel P} (used to define the Rieffel projection in $\mathcal{A}_\theta$). In particular, $e^{\theta}(u, v)$ is a projection if $u$ and $v$ satisfy $uv=e^{2\pi i\theta}vu.$
\end{definition}
It is clear that $e^{\theta}(u, v)$ is always self-adjoint.

\begin{proposition}\label{P_RieffelEgap}
Let $\theta \in (0, 1)$. There exists a $\delta > 0$ such that, for any unital $C^*$-algebra $A$,
any pair of unitaries $u, v$ in $A$, if $u$ and $v$ satisfy 
$\|uv - e^{2\pi i \theta}vu\| < \delta$,
then
\[
\|e^{\theta}(u, v)^2 - e^{\theta}(u, v)\| < \frac{1}{4}.
\]
In particular, ${\rm sp}(e^{\theta}(u, v))$ has a gap at $\frac{1}{2}$.
\end{proposition}
\begin{proof}
Suppose the statement is false. Let $(\delta_m)_{m=1}^{\infty}$ be a sequence of positive numbers decreasing to $0$.
Then for any positive integer $m$, there exists a unital $C^*$-algebra $A_m$,
a pair of unitaries $\tilde{u}_m, \tilde{v}_m$ in $A_m$ such that $ \|\tilde{u}_m \tilde{v}_m - e^{2\pi  i \theta} \tilde{v}_m\tilde{u}_m\| < \delta_m$,
but
\[
\|e^{\theta}(\tilde{u}_m, \tilde{v}_m)^2 - e^{\theta}(\tilde{u}_m, \tilde{v}_m)\| \geq \frac{1}{4}.
\]
Let $B = \prod_{m=1}^{\infty} A_m \slash \oplus_{m=1}^{\infty} A_m$.
Let $\pi \colon \prod_{m=1}^{\infty} A_m \rightarrow B$ be the canonical quotient map.
Let $\tilde{u}=(\tilde{u}_1, \tilde{u}_2, \dots)$ and $\tilde{v} = (\tilde{v}_1, \tilde{v}_2, \dots)$. Then $\pi(\tilde{u}), \pi(\tilde{v})$ sastifies
the exact rotation relation $\pi(\tilde{u})\pi(\tilde{v})=e^{2\pi i\theta}\pi(\tilde{v})\pi(\tilde{u})$, hence there is a homomorphism
\[
\phi \colon \mathcal{A}_\theta \rightarrow B, \quad \phi(u_1) = \pi(\tilde{u})\,\,{\rm and}\,\,\phi(u_2) = \pi(\tilde{v}),
\]
where $u_1,u_2$ are the canonical generators of $\mathcal{A}_\theta$.
In particular, the element
\[
\pi\big(e^{\theta}(\tilde{u}_1, \tilde{v}_1), e^{\theta}(\tilde{u}_2, \tilde{v}_2), \dots \big) = \pi(e(\tilde{u}, \tilde{v})) = \phi(e(u_1,u_2))
\]
is a projection. But this implies that
\[
\lim_{m \rightarrow \infty} \|e^{\theta}(\tilde{u}_m, \tilde{v}_m)^2 - e^{\theta}(\tilde{u}_m, \tilde{v}_m)\| = 0,
\]
which is a contradiction.
\end{proof}

\begin{definition}\label{D_RieffelE}
Let $\theta \in (0, 1)$. Let $\delta > 0$ be chosen as in Proposition \ref{P_RieffelEgap}. 
Let $A$ be a unital $C^*$-algebra
and let $u, v$ be a pair of unitaries in $A$ such that $ \|uv - e^{2\pi i \theta}vu\| < \delta$.
Let $\chi_{(\frac{1}{2}, \infty)}$ be the characteristic function on $(\frac{1}{2}, \infty)$.
We define $R^{\theta}(u, v) =\chi_{(\frac{1}{2}, \infty )}(e^{\theta}(u, v))$ and call it the \emph{Rieffel element} of $u$ and $v$ with respect to $\theta$.
\end{definition}

The following is a straightforward application of functional calculus:
\begin{proposition}\label{P_RieffelEConst}
Let $\theta \in (0, 1)$. Let $\delta > 0$ be chosen as in Proposition \ref{P_RieffelEgap}. 
Let $u, v$ be a pair of unitaries in a unital C*-algebra $A$
such that $ \|uv - e^{2\pi i \theta}vu\| < \delta$.
Then $b_{u, v}^{\theta} = [R^{\theta}(u, v)]$. In particular, if $\ep < \delta$,
then $b_{\ep}^{\theta} = [R^{\theta}(\mathfrak{u}_{\varepsilon,\theta}, \mathfrak{v}_{\varepsilon,\theta})]$
in $K_0(\mathcal{S}_{\ep, \theta})$.
\end{proposition}

\begin{lemma}\label{P_RieffelProjsame}
Let $u$ and $v$ be two unitaries in a unital C*-algeba $A$. 
Suppose $\|uv-e^{2\pi i\theta_1}vu\|< 2$. Then there is 
a $\delta > 0$ such that, whenever $\theta_2$ is a real number
such that $\| \theta_2 - \theta_1\| < \delta$, we have
$\|uv-e^{2\pi i\theta_2}vu\|< 2$ 
and 
$b^{\theta_1}_{u,v}=b^{\theta_2}_{u,v}$ in $K_0(A).$
\end{lemma}
\begin{proof}  
Let $\|uv-e^{2\pi i\theta_1}vu\|=\varepsilon_1 < 2$. 
Choose $\delta > 0$ so that 
$|e^{2\pi i\theta_1}-e^{2\pi i(\theta_1 + \delta)}| <\frac{2 - \varepsilon_1}{6}$. 
Let $\theta_2$ be a real number such that $\| \theta_2 - \theta_1 \| < \delta$. Then 
$$\|uv-e^{2\pi i\theta_2}vu\|\leq\|uv-e^{2\pi i\theta_1}vu\|+|e^{2\pi i\theta_1}-e^{2\pi i\theta_2}|\leq\varepsilon_1 +\frac{2 - \varepsilon_1}{6}< 2.$$
Let $\varepsilon_2=\varepsilon_1 +\frac{2 - \varepsilon_1}{6}.$
Let $\theta = \frac{\theta_1 + \theta_2}{2}$.
We can compute that
\begin{align*}
\|\mathfrak{u}_{\varepsilon_1,\theta_1}\mathfrak{v}_{\varepsilon_1,\theta_1}-e^{2\pi i\theta}\mathfrak{v}_{\varepsilon_1,\theta_1}\mathfrak{u}_{\varepsilon_1,\theta_1}\|
& \leq \|\mathfrak{u}_{\varepsilon_1,\theta_1}\mathfrak{v}_{\varepsilon_1,\theta_1}-e^{2\pi i\theta_1}\mathfrak{v}_{\varepsilon_1,\theta_1}\mathfrak{u}_{\varepsilon_1,\theta_1}\|+|e^{2\pi i\theta_1}-e^{2\pi i\theta}|\\
&\leq \varepsilon_1 + \frac{2 - \varepsilon_1}{6} = \frac{2 + 2\varepsilon_1}{6}.
\end{align*}
and 
\begin{align*}
\|\mathfrak{u}_{\varepsilon_2,\theta_2}\mathfrak{v}_{\varepsilon_2,\theta_2}-e^{2\pi i\theta}\mathfrak{v}_{\varepsilon_2,\theta_2}\mathfrak{u}_{\varepsilon_2,\theta_2}\|
& \leq \|\mathfrak{u}_{\varepsilon_2,\theta_2}\mathfrak{v}_{\varepsilon_2,\theta_2}-e^{2\pi i\theta_2}\mathfrak{v}_{\varepsilon_2,\theta_2}\mathfrak{u}_{\varepsilon_2,\theta_2}\|+|e^{2\pi i\theta_2}-e^{2\pi i\theta}|\\
&\leq \varepsilon_2 + \frac{2 - \varepsilon_1}{6} = \frac{2 + 2\varepsilon_1}{3}.
\end{align*}


 Let $\ep = \frac{2 + 2\varepsilon_1}{3} < 2$. We have two surjective homomorphisms 
 \[
 \phi_{1} \colon \mathcal{S}_{\ep,\theta} \rightarrow \mathcal{S}_{\varepsilon_1,\theta_1}, \quad
 \phi_{1}(\mathfrak{u}_{\ep,\theta})=\mathfrak{u}_{\varepsilon_1,\theta_1},\,\,
 \phi_{1}(\mathfrak{v}_{\ep,\theta})=\mathfrak{v}_{\varepsilon_1,\theta_1}
 \]
 and
 \[
 \phi_{2} \colon \mathcal{S}_{\ep,\theta} \rightarrow \mathcal{S}_{\varepsilon_2,\theta_2}, \quad
 \phi_{2}(\mathfrak{u}_{\ep,\theta})=\mathfrak{u}_{\varepsilon_2,\theta_2},\,\,
 \phi_{2}(\mathfrak{v}_{\ep,\theta})=\mathfrak{v}_{\varepsilon_2,\theta_2}.
 \]

We use the same notation as in Definition \ref{dbotA}. Then
\beq
b^{\theta_1}_{u,v}&=&(\phi_{u,v}^{\theta_1})_*(b_{\varepsilon_1}^{\theta_1})=(\phi_{u,v}^{\theta_1})_*\circ \phi_{1 *}(b_{\ep}^{\theta})\nonumber\\
&=&(\phi_{u,v}^{\theta})_*(b_{\ep}^{\theta}) =(\phi_{u,v}^{\theta_2})_*\circ \phi_{2 *}(b_{\ep}^{\theta})\nonumber\\
&=&(\phi_{u,v}^{\theta_2})_*(b_{\varepsilon_2}^{\theta_2}) = b^{\theta_2}_{u,v}.\nonumber
\eneq
\end{proof}

\begin{definition}\label{logarithm}
Let $\theta \in [0, 1)$. Denote by $\log_{\theta}$ the continuous branch of logarithm defined on $F_{\theta}=\{e^{it}:t\in (2\pi\theta-\pi,2\pi\theta+\pi)\}$ with values in $\{ri:r\in (2\pi\theta-\pi,2\pi\theta+\pi)\}$ such that $\log_{\theta}(e^{2\pi i \, \theta}) = 2\pi i \, \theta$.
  Note that if $u$ is any unitary in some $C^*$-algebra $A$ such that 
  $\|u-e^{2\pi i \theta}\|<2$, then ${\rm sp}(u)$ has a gap at $e^{2\pi i \theta + \pi i}$, thus $\log_{\theta}(u)$
  is well-defined. In particular, if $\theta=0,$ we simply write $\log(u)$ for $\log_{0}(u).$
\end{definition}

Now we are ready to state our generalized version of the Exel trace formula. The Exel trace formula is a very convenient tool. In
fact, the recent development in the connection to the Elliott program of classification
of amenable $C^*$-algebras shows that the Exel trace formula has many further
applications. The Exel trace formula brought together the bott element, a topological
obstruction, and rotation number, a dynamical description.
The case where $A$ is the matrix algebra and $ \theta = 0$ is proved in \cite{Exel} (Lemma 3.1 and Theorem 3.2 of \cite{Exel}). The case where $A$ is an arbitrary unital $C^*$-algebra and $\theta = 0$ is proved in Theorem 3.7  of \cite{Hua-Lin}.
\begin{theorem}[\textbf{Generalized Exel trace formula}]\label{trace formula} Let $A$ be a unital $C^*$-algebra with $T(A)\neq \emptyset$, and $\theta \in [0, 1)$.
Then for any $u,v\in U(A)$ and $\|uv-e^{2\pi i \theta}vu\|< 2$, let $b_{u,v}^{\theta}$ be defined as in Definition \ref{dbotA}, we have
\[
\rho_A( b_{u,v}^{\theta})(\tau)=\frac{1}{2\pi i}\tau(\log_{\theta}(uvu^*v^*))\quad \mbox{for all}\quad \tau\in T(A).
\]
\end{theorem}

\begin{proof}
It is clear that we need only to prove this theorem for $A$
being the soft rotation algebra $\mathcal{S}_{\varepsilon,\theta}$.
The general case is proved by considering
the canonical homomorphism from $\mathcal{S}_{\varepsilon,\theta}$ to $A$
and use functional calculus.
The proof follows by essentially the same argument as
in Theorem 3.5 of \cite{Hua-Lin} with a small modification.
For the reader's convenience, we sketch a proof here.

First of all, by Proposition \ref{P_RieffelProjsame},
the quantity $\rho_{A}( b_{u,v}^{\theta})(\tau)=\tau(b_{u,v}^{\theta})$ on the left-hand side  remains unchanged for small perturbations of $\theta$. The same is true for  quantity on the right-hand side.
So without loss of generality, we can assume $\theta$ is rational
and not equal to $0$.

By exactly the same argument as in the proof of Theorem 3.5 of \cite{Hua-Lin},
for each tracial state $\tau$ of $A$,
there is an integer $k_\tau$ such that
\begin{equation}\label{E_TraceF1}
\tau( b_{\varepsilon}^{\theta})=\frac{1}{2\pi i}\tau(\log_{\theta}(\mathfrak{u}_{\varepsilon,\theta}\mathfrak{v}_{\varepsilon,\theta}\mathfrak{u}^*_{\varepsilon,\theta}\mathfrak{v}^*_{\varepsilon,\theta}))+k_\tau.
\end{equation}
Moreover, there is an integer $K$ such that $|k_\tau| \leq K$, for all tracial state $\tau$ on $A$.
The rest of the effort is then devoted to prove that they are actually $0$.
If $\theta$ is rational, then $\mathcal{A}_\theta$ is
strong Morita equivalent to $C(\T^2)$, by Corollary 4.2 of \cite{Elliott-Lihanfeng}.
In particular, there is an irreducible finite-dimensional representation
$\pi_0 \colon \mathcal{A}_{\theta} \rightarrow M_n$.
Let $\psi \colon \mathcal{S}_{\varepsilon,\theta} \rightarrow \mathcal{A}_\theta$
be the canonical surjective homomorphism. Let
\[
\iota \colon M_n \rightarrow M_n(\mathcal{S}_{\varepsilon,\theta})
\cong M_n \otimes \mathcal{S}_{\varepsilon,\theta}, \quad \iota(a) = a \otimes 1_{\mathcal{S}_{\varepsilon,\theta}}, \,\,
\forall a \in M_n.
\]

Define $\pi = \iota \circ \pi_0 \circ \psi$.
For any positive integer $s$, let
$\Phi \colon \mathcal{S}_{\varepsilon,\theta} \rightarrow M_{ns+1}(\mathcal{S}_{\varepsilon,\theta})$
be the unital homomorphism defined by
\[
\Phi(a) = \mathrm{diag}\{a, \pi(a), \pi(a), \dots, \pi(a)\}, \text{\,\,for\,\,all\,\,} a \in \mathcal{S}_{\varepsilon,\theta}.
\]

For any positive integer $k$, let $\mathrm{Tr}$ be the unnormalized trace on $M_k$.
Let $\tau$ be an arbitrary tracial state on $\mathcal{S}_{\varepsilon,\theta}$,
we use $\tau^{\oplus k}$ to denote the trace on $M_k \otimes \mathcal{S}_{\varepsilon,\theta}$
induced by $\tau^{\oplus k}(a \otimes b) = \mathrm{Tr}(a)\tau(b)$.
Let $\tau_0 = \frac{1}{ns+1}\tau^{\oplus ns+1}$,
which is a tracial state on $M_{ns+1}(\mathcal{S}_{\varepsilon,\theta})$.
Then $\tau_0 \circ \Phi$ is a tracial state on $\mathcal{S}_{\varepsilon,\theta}$.
Therefore equation (\ref{E_TraceF1}) gives
\[
\tau_0 \circ \Phi( b_{\varepsilon}^{\theta})=
\frac{1}{2\pi i}\tau_0\circ \Phi(\log_{\theta}(\mathfrak{u}_{\varepsilon,\theta}\mathfrak{v}_{\varepsilon,\theta}\mathfrak{u}^*_{\varepsilon,\theta}\mathfrak{v}^*_{\varepsilon,\theta}))+k_{\tau_0 \circ \Phi}.
\]

Let $b_{u, v}^{\theta}$ be the Rieffel projection in $\mathcal{A}_\theta$. From the definition of $\Phi$, we can compute that
\begin{align*}
\tau_0 \circ \Phi( b_{\varepsilon}^{\theta}) 
&= \frac{1}{ns+1}\left(\tau(b_{\varepsilon}^{\theta}) + s\tau^{\oplus n}(\pi(b_{\varepsilon}^{\theta})) \right) \\
&= \frac{1}{ns+1}\left(\tau(b_{\varepsilon}^{\theta}) + 
s\mathrm{Tr}(\pi_0\circ \psi(b_{\varepsilon}^{\theta}))\right)\\
&= \frac{1}{ns+1}\left(\tau(b_{\varepsilon}^{\theta}) + 
s\mathrm{Tr}(\pi_0(b_{u, v}^{\theta}))\right)\\
&=\frac{1}{ns+1}\left(\tau(b_{\varepsilon}^{\theta}) + 
ns\theta\right).
\end{align*}

A similar calculation shows that
\[
\tau_0\circ \Phi(\log_{\theta}(\mathfrak{u}_{\varepsilon,\theta}\mathfrak{v}_{\varepsilon,\theta}\mathfrak{u}^*_{\varepsilon,\theta}\mathfrak{v}^*_{\varepsilon,\theta}))
=\frac{1}{ns+1}\left(
\tau(\log_{\theta}(\mathfrak{u}_{\varepsilon,\theta}\mathfrak{v}_{\varepsilon,\theta}\mathfrak{u}^*_{\varepsilon,\theta}\mathfrak{v}^*_{\varepsilon,\theta})) +2\pi i ns\theta
\right).
\]

This implies that
\[
|k_{\tau_0 \circ \Phi}| = |\frac{k_{\tau}}{ns+1}| \leq \frac{K}{ns+1}.
\]
Since $k_{\tau_0 \circ \Phi}$ is always an integer, it must be $0$ if $s$ is large enough. This implies
that $k_\tau$ is $0$.
\end{proof}

As an immediate application of the Exel trace formula, we show that certain trace condition is necessary
to prove stability of rotation relations.
\begin{cor}\label{C_ETF}
Let $\theta \in [0, 1)$. There exists a $\delta > 0$ such that, for any unital $C^*$-algebra $A$ with $T(A)\neq \emptyset$,
any unitaries $u, v, \tilde{u}, \tilde{v}$ in $A$ satisfying the following:
\begin{enumerate}
    \item $\|uv - e^{2\pi i \theta}vu\| < 2,$
    \item $\tilde{u}\tilde{v} = e^{2\pi i \theta}\tilde{v}\tilde{u},$ and
    \item $\|\tilde{u} - u\| < \delta, \, \|\tilde{v} - v\| < \delta$,
\end{enumerate}
we have $\frac{1}{2\pi i} \tau(\log_\theta(uvu^*v^*)) = \theta$ for any tracial state $\tau$ on $A$.
\end{cor}

\begin{proof}
We first assume $\theta \neq 0$.
Choose $\delta_0$ according to Proposition \ref{P_RieffelEConst}. 
Since $R^{\theta}(u, v)$ is defined in terms of continuous 
functional calculus of $u, v$, there is a 
$\delta_1$ such that whenever $\|\tilde{u} - u\| < \delta_1$
and $\|\tilde{v} - v\| < \delta_1$, we have 
$\| R^{\theta}(u, v) - R^{\theta}(\tilde{u}, \tilde{v})\| < \frac{1}{2}$. Let $\delta = \min \{\delta_0, \delta_1\}$. 
Then for any $u, v, \tilde{u}, \tilde{v}$ satifies conditions 
(1), (2), (3) in the statement of this corollary, we have
\[
b_{u,v}^{\theta} = [R^{\theta}(u, v)] = [R^{\theta}(\tilde{u}, \tilde{v}] = b_{\tilde{u}, \tilde{v}}^{\theta}.
\]
Hence 
\begin{align*}
\frac{1}{2\pi i} \tau(\log_\theta(uvu^*v^*))
& = \rho_A(b_{u,v}^{\theta}) (\tau)
  = \rho_A(b_{\tilde{u}, \tilde{v}}^{\theta})(\tau)\\
& = \frac{1}{2\pi i} \tau(\log_\theta(\tilde{u}\tilde{v}\tilde{u}^*\tilde{v}^*)) = \theta
\end{align*}
for all $\tau\in T(A).$
The proof for the case that $\theta = 0$ is exactly the same, one just
have to use the bott element instead of the Rieffel element.
\end{proof}






\section{Stability of rotation relations in $C^*$-algebras of tracial rank at most one} \label{S_SofTR1}

Let us begin with a brief outline of our strategy of proving stability of rotation relations.
Let $\Theta$ be a $n \times n$ real skew-symmetric matrix. Suppose $ \{ v_i \}_{i = 1, 2, \dots n}$
are $n$ unitaries in a unital simple $C^*$-algebra $A$ with $TR(A) \leq 1$ 
which almost satisfy the rotation relation with respect to $\Theta$.
Then there is an almost homomorphism from $\mathcal{A}_\Theta$ to $A$.
Now stability of the rotation relation is equivalent to that 
this almost homomorphism is close to an actual homomorphism.

The later problem is usually divided into two parts:
the existence part and the uniqueness part.
An almost homomorphism will induce an `almost'  homomorphism 
between the invariants of the two $C^*$-algebras, where the invariant consists of
the K-theories and the trace spaces together with certain comparability maps.
It is relatively  easy to show that an `almost' homomorphism 
of the invariants is close to an actual homomorphism.
The existence part says that a homomorphism at the invariant level
lifts to a homomorphism at the $C^*$-algebra level.
The uniqueness part says that, two almost homomorphisms which induces `almost'
the same maps on the invariants are almost unitarily equivalent. 
Therefore, conjugating suitable unitaries, one shows that an almost homomorphism is close
to an actual homomorphism.

It turns out that, to implement the above strategy, if $n \geq 3$, the current available tools
require an additional injectivity condition. We give the following counterexample 
to the stability problem which shows that, without an injectivity condition, 
the stability problem can be very complicated.

\begin{example}
There is a sequence of triples  $( U_{n, 1}, U_{n, 2}, U_{n,3})_{n \in \N}$ of unitary matrices
such that
\begin{align*}
& \lim_{ n \rightarrow \infty} \| U_{n, j}U_{n,k} - U_{n, k}U_{n,j} \| = 0, \quad 1 \leq j, k \leq 3\,\,\,{\rm and}\\
& \frac{1}{2\pi i} \tau (\log(U_{n, j}U_{n,k}U_{n, j}^*U_{n,k}^*)) = 0, \quad 1 \leq j, k \leq 3,\, n\in \N,
\end{align*}
where $\log = \log_0,$
but there is no sequence of triples $(U'_{n,1}, U'_{n,2}, U'_{n,3})_{n\in \N}$ of exactly commuting
unitary matrices  
$$U_{n,j}'U_{n,k}'=U_{n,k}'U_{n,j}',\quad 1 \leq j, k \leq 3$$
such that
\[
\lim_{n \rightarrow \infty} \|U_{n, j} - U'_{n,j}\| = 0,\,\, j = 1, 2, 3.
\]
\end{example}

\begin{proof}
By Voiculescu's result \cite{Voiculescu1},
there  exists a sequence of triples $( H_{n, 1}, H_{n, 2}, H_{n,3})_{n \in \N}$ of self-adjoint matrices with norm one such that
\[
\lim_{ n \rightarrow \infty} \| H_{n, j}H_{n,k} - H_{n, k}H_{n,j} \| = 0, \quad 1 \leq j, k \leq 3
\]
for which does not exist a sequence of triples $(H'_{n,1}, H'_{n,2}, H'_{n,3})_{n\in \N}$ of exactly commuting self-adjoint
matrices such that
\[
\lim_{n \rightarrow \infty} \|H_{n, j} - H'_{n,j}\| = 0,\,\, j = 1, 2, 3.
\]
Let $U_{n, j} = e^{\pi i\, H_{n, j}\slash 2}$. Then 
\[
\lim_{ n \rightarrow \infty} \| U_{n, j}U_{n,k} - U_{n, k}U_{n,j} \| = 0, \quad 1 \leq j, k \leq 3.
\]
Now fix a pair $j, k$. By Lin's result \cite{Lin2}, passing to subsequences if necessary,
there is a sequence
of pairs $(H''_{n,j}, H''_{n,k})_{n\in \N}$ of commuting self-adjoint matrices 
 such that
\[
\lim_{n \rightarrow \infty} \|H_{n, j} - H''_{n,j}\| = 0\,\,\,{\rm and}\,\,\, \lim_{n \rightarrow \infty} \|H_{n, k} - H''_{n,k}\| = 0.
\]
Therefore
\[
\lim_{n \rightarrow \infty} \|U_{n, j} - e^{\pi i \, H''_{n,j}\slash 2}\| = 0\,\,\,{\rm and}\,\,\, \lim_{n \rightarrow \infty} \|U_{n, k} - e^{\pi i \, H''_{n,k}\slash 2}\| = 0.
\]
 By Corollary \ref{C_ETF}, if $n$ is sufficiently large, then
\[
\frac{1}{2\pi i} \tau (\log (U_{n, j}U_{n,k}U_{n, j}^*U_{n,k}^*)) = 0.
\]
We claim that there is no sequence of triples $(U'_{n,1}, U'_{n,2}, U'_{n,3})_{n\in \N}$ of commuting unitaries such that
\[
\lim_{n \rightarrow \infty} \|U_{n, j} - U'_{n,j}\| = 0,\,\, j = 1, 2, 3.
\]
Otherwise, since the spectra  of $U_{n,1},U_{n,2},U_{n,3}$ are all in $\{e^{\pi i\theta}:\theta\in [-\frac{1}{2},\frac{1}{2}]\}$  for all $n\in \N$, and
\begin{eqnarray}\label{limit}
\lim_{n \rightarrow \infty} \|U_{n, j} - U'_{n,j}\| = 0,\,\, j = 1, 2, 3,\end{eqnarray} we can find a sufficiently large integer $N$ such that for any $n\geq  N$,  $-1$ is not in the spectra of $U_{n,1}',U_{n,2}',U_{n,3}'$. Without loss of generality, we can assume that $-1$ is not in the spectra of $U_{n,1}',U_{n,2}',U_{n,3}'$ for $n\in \N.$ Let $H'_{n,j} =\frac{2}{\pi i} \log (U'_{n,j})$ for $j=1, 2, 3$.
Then $(H'_{n,1}, H'_{n,2}, H'_{n,3})_{n\in \N}$ is a sequence
of commuting triples of self-adjoint matrices. Note that $H_{n,j} =\frac{2}{\pi i} \log (U_{n,j})$ for $j=1, 2, 3$. By
continuous functional calculus and (\ref{limit}), we get
\[
\lim_{n \rightarrow \infty} \|H_{n, j} - H'_{n,j}\| = 0,
\]
which is a contradiction.
\end{proof}

The existence theorem and the uniqueness theorem are taken from \cite{Lin17}.
Before we can state them, let us introduce some notation.
Let $C = PM_n(C(X))P$ be a homogenous algebra, where $X$ is a compact metric space
and $P$ is a projection in $M_n(C(X))$. Then for any tracial state $\tau$ on $C$,
there is a probability measure $\mu$ on $X$ such that 
\[
\tau(f) = \int_{x \in X} \mathrm{tr}_x(f(x))\,d\mu,
\]
where $\mathrm{tr}_x$ is the normalized trace on the fiber corresponds to $x$ (which is isomorphic to a matrix algebra).
We shall use $\mu_\tau$ to denote the measure associated to $\tau$.

Let $C$ be a unital $C^*$-algebra with $T(C) \neq \emptyset$. 
By Theorem 3.2 of \cite{Thomsen}, the de la Harpe-Scandalis determinant provides a continuous homomorphsim
\begin{equation}\label{E_HSdet}
\bar{\Delta} \colon U_0(M_{\infty}(C))/CU(M_{\infty}(C))
\rightarrow \mathrm{Aff}(T(C))/\overline{\rho_C(K_0(C))}.
\end{equation}
By Corollary 3.3 of \cite{Thomsen}, there is an induced split exact sequence

\begin{equation} \label{E_Thomsen_es}
0 \rightarrow \mathrm{Aff}(T(C)) / \overline{\rho_C(K_0(C))}
\rightarrow U(M_{\infty}(C))/CU(M_\infty(C)) 
\xrightarrow[]{\pi_C} K_1(C) \rightarrow 0. 
\end{equation}
The reader is referred to \cite{Thomsen} for more details 
of the homomorphism (\ref{E_HSdet}) and the exact sequence
(\ref{E_Thomsen_es}). 

Since the exact sequence (\ref{E_Thomsen_es}) is split,
there is a homomorphism 
\[
J_C \colon K_1(C) \rightarrow U(M_{\infty}(C))/CU(M_\infty(C))
\]
such that $\pi_C \circ J_C = \mathrm{Id}_{K_1(C)}$. 
We shall use $U_c(K_1(C))$ to denote a set of representatives 
of $J_C(K_1(C))$ in $U(M_{\infty}(C))$. 
For each $C$, we shall fix a splitting map $J_C$ and a set of representatives $U_c(K_1(C))$
if not mentioned explicitly.

For any two unitaries $u, v \in U(M_{\infty}(C))$ such that 
$uv^* \in U_0(M_{\infty}(C))$, define
\[
\mathrm{dist}(\bar{u}, \bar{v}) = \| \bar{\Delta}(uv^*)\|,
\]
where the norm is the quotient norm on 
$\mathrm{Aff}(T(C))/\rho_C(K_0(C))$.

The following lemma allows us to estimate the norm we just defined:
\begin{lemma} \label{L_Dist}
For any $\ep > 0$, there is a $\delta > 0$ such that, whenever $A$ is a unital
$C^*$-algebra, if $u, v$ are unitaries in $U_{\infty}(A)$ such that $\|wuv^* - 1\| < \delta$
for some $w \in CU_{\infty}(A)$, then
\[
\mathrm{dist} (\bar{u}, \bar{v}) < \ep.
\]
\end{lemma}

\begin{proof}
Since $\bar{\Delta}(w) \in \rho_A(K_0(A))$ for $w \in CU_{\infty}(A)$, we can ignore $w$ 
when calculate the distance. So without loss of generality we assume $w = 1$ and $\ep< 1$. 
Let $\delta = |e^{2 \pi i \ep} - 1|$. Suppose $\|uv^* -1\| < \delta$. 
Then the spectrum of $uv^*$ is contained in $\{e^{2\pi i \theta} \,\vert\, \theta \in (-\ep, \ep)\}$.
Therefore a continuous (normalized) logarithm $\log$ can be defined:
\[
\log \colon \mathrm{sp}(uv^*) \rightarrow (-\ep2\pi i, \ep2\pi i), \quad \log(e^{2\pi i \theta}) =2\pi i \theta.
\]
Let $h = \frac{1}{2\pi i}\log (uv^*)$. Then $\|h\| \leq \ep$ and $uv^* = e^{2\pi i h}$.
Let $\eta \colon [0, 1] \rightarrow U_{\infty, 0}(A)$ be defined by $\eta(t) = e^{2 \pi i th}$.
Then for any tracial state $\tau$ on $A$, we have
\[
|\bar{\Delta} ( \eta )(\tau)| =| \frac{1}{2 \pi i} \int_0^1 \tau (\eta'(t)\eta(t)^*)\,{\rm d}t|
= |\frac{1}{2\pi i} \int_0^1 \tau (2 \pi i h)\,{\rm d}t |\leq \| h \| < \ep.
\]
Therefore $\mathrm{dist} (\bar{u}, \bar{v}) = \| \bar{\Delta} (u v^*) \| < \ep$.
\end{proof}

\begin{definition} Let $X$ be a compact metric space, let $x\in X$ and let $r > 0$. Denote by $O_r(x)$
the open ball with center at $x$ and radius $r$. If $x$ is not specified, $O_r$ is an open ball of radius $r$.
\end{definition}

Now we are ready to  state the uniqueness theorem and existence theorem.
\begin{theorem}\label{T_unique}
(Theorem 5.3 and Corollary 5.5 of \cite{Lin17})
Let $C = PM_n(C(X))P$, where $X$ is a compact metric space and $P$ is a projection in $M_n(C(X))$. 
Let $\Delta \colon (0, 1) \rightarrow (0, 1)$
be a non-decreasing function such that $\lim_{r \rightarrow 0} \Delta(r) = 0$.
Let $\ep > 0$ and
let $\mathcal{F} \subset C$ be a finite subset. 
Then there exists $\eta> 0$, $\delta > 0$, 
a finite subset $\mathcal{G} \subset C$,
a finite subset $\mathcal{P} \subset \underline{K}(C)$, 
a finite subset $\mathcal{H} \subset C_{s.a.}$ and 
a finite subset $\mathcal{U} \in U_c(K_1(C))$ satisfying the following:

Suppose that $A$ is a unital separable simple $C^*$-algebra with $TR(A) \leq 1$.
Suppose $L_1, L_2 \colon C \rightarrow A$ are two unital $\mathcal{G}$-$\delta$-multiplicative completely positive contractive
(c.p.c.) maps such that
\begin{align}
[L_1]\vert_\mathcal{P} & = [L_2]\vert_{\mathcal{P}},\phantom{aaaaaaaaaaaaaaaaaaaaa}\label{unique1}\\
|\tau \circ L_1(g) - \tau \circ L_2(g)| & < \delta,
\quad \text{for all\,\,} g \in \mathcal{H} \\
\mathrm{dist}(\overline{ \langle L_1(u) \rangle }, \overline{\langle L_2(u) \rangle}) & < \delta,\quad \text{for all\,\,} u \in \mathcal{U}\quad \text{and}\\
\mu_{\tau \circ L_i}(O_r) & > \Delta(r), \quad i = 1, 2,
\end{align}
for all $\tau \in T(A)$ and  all open balls $O_r$ in $X$ with radius $r\geq  \eta$. Then there exists a unitary $W \in A$ such that
\[
\| \mathrm{Ad} W \circ L_1 (f) - L_2(f)\| < \ep 
\text{\,\,for all\,\,} f \in \mathcal{F}.
\]
\end{theorem}

\begin{definition} 
Let $A, C$ be two unital $C^*$-algebras. Denote by $KL(C,A)^{++}$ the set of those $\kappa\in \Hom _{\Lambda}(\underline{K}(C),\underline{K}(A))$ such that
$$\kappa(K_0(C)_+\backslash \{0\})\subset K_0(A)_+\backslash \{0\}.$$

Denote by $KL_e(C, A)^{++}$ the set of those $\kappa\in  KL(C, A)^{++}$ such that $\kappa([1_C]) = [1_A].$
\end{definition}

\begin{Notation}
Let $h \colon C \rightarrow A$ be a unital homomorphism. Denote by $[h]$ the induced homomorphism on the K-groups.
Denote by $h_\sharp$ the induced map on the tracial state spaces. Denote by $h^{\ddagger}$ the induced map
\[
h^{\ddagger} \colon  U(M_{\infty}(C))/CU(M_{\infty}(C)) \rightarrow U(M_{\infty}(A))/CU(M_{\infty}(A)).
\]
\end{Notation}

Note that if $A$ is a simple unital separable $C^*$-algebra with tracial rank no more than one,
then it has stable rank one (see Theorem 3.6.10 of \cite{Lin-book}) and  there is an isomorphism (see Corollary 3.5 of \cite{Lin-JFA-2010})
\[
U(A)/CU(A) \cong U(M_{\infty}(A))/CU(M_{\infty}(A)).
\]

\begin{definition} 
Let $\kappa \in KL_e(C, A)^{++}$ and let $\lambda : T(A) \rightarrow  T(C)$ be a continuous affine map.
We say that $\lambda$ is compatible with $\kappa$ if $\lambda$ is compatible with $\kappa|_{K_0(C)}.$
(i.e. $\rho_C([p])(\lambda(\tau)) = \rho_{A}(\kappa([p]))(\tau)$, for any projection $p \in M_{\infty}(C)$ and any tracial state $\tau\in T(A)$).
Let
\[
\alpha \colon U(M_{\infty}(C))/CU(M_{\infty}(C)) \rightarrow U(A)/CU(A)
\]
be a continuous homomorphism. Then by the exact sequence (\ref{E_Thomsen_es}),
there is an induced map $\af_1 \colon K_1(C) \rightarrow K_1(A)$.
We say $\alpha$ and $\kappa$ are compatible, if $\alpha_1 = \kappa|_{K_1(C)}$.
We say $\alpha, \lambda$ and $\kappa$ are compatible if $\lambda, \kappa$ are compatible, $\alpha, \lambda$ are compatible
and $\alpha, \kappa$ are compatible.
\end{definition}

\begin{theorem}(Theorem 6.11 of \cite{Lin17})\label{T_exist}
Let $C$ be a unital AH-algebra and let $A$ be a unital infinite
dimensional separable simple $C^*$-algebra with $TR(A) \leq 1$.
Then for any $\kappa \in KL_e(C, A)^{++}$, 
any affine continuous map $\gamma \colon T(A) \rightarrow T_f(C)$
and any continuous homomorphism 
\[
\alpha \colon U(M_{\infty}(C))/CU(M_{\infty}(C)) \rightarrow U(A)/CU(A)
\]
such that $\kappa, \gamma$ and $\alpha$ are compatible, there is a unital monomorphism
$h \colon C \rightarrow A$ such that
\[
[h] = \kappa, \quad h_\sharp = \gamma \quad \text{and\quad} 
h^{\ddagger} = \alpha.
\]
\end{theorem}

\begin{proposition}\label{P_RAisAH}
For any $n \times n$ real skew-symmetric matrix  $\Theta$, 
there are $C_m$ which are direct sum  of homogeneous algebras whose $K_0$ and $K_1$ groups are finitely generated free abelian groups, for $m=1,2,\dots,$ such that $\mathcal{A}_{\Theta}= \underrightarrow{\lim} \,C_m.$  In particular,
$\mathcal{A}_{\Theta}$ is an $AH$-algebra. 

\end{proposition}

\begin{proof}
By Proposition 3.3 of \cite{Elliott-Lihanfeng} and 
Theorem 1.1 of \cite{Li-2004}, there is an integer $l$
with $0 \leq l \leq n$ and a non-degenerate $l \times l$
 real skew-symmetric matrix $\widetilde{\Theta}$ such that
 $\mathcal{A}_{\Theta}$ is strongly Morita equivalent to 
 $\mathcal{A}_{\widetilde{\Theta}} \otimes C(\T^{n-l})$.
So there is an isomorphism
\[
\psi \colon \mathcal{A}_{\Theta} \otimes \mathcal{K} \rightarrow \mathcal{A}_{\widetilde{\Theta}} \otimes C(\T^{n-l}) \otimes \mathcal{K},
\]
where $\mathcal{K}$ is the algebra of compact operators.
Then
\[
\mathcal{A}_{\Theta}=
(1_{\mathcal{A}_{\Theta}}\otimes e_{11})(\mathcal{A}_{\Theta}\otimes \mathcal{K})(1_{\mathcal{A}_{\Theta}}\otimes e_{11})
\cong \psi(1_{\mathcal{A}_{\Theta}}\otimes e_{11}) (\mathcal{A}_{\widetilde{\Theta}} \otimes C(\T^{n-l})\otimes \mathcal{K})\psi(1_{\mathcal{A}_{\Theta}}\otimes e_{11}).
\]
We can find a projection $P\in M_{N}(\mathcal{A}_{\widetilde{\Theta}} \otimes C(\T^{n-l}))$ which is equivalent to $\psi(1_{\mathcal{A}_{\Theta}}\otimes e_{11})$ for some
$N \in \mathbb{N}.$ So $\mathcal{A}_{\Theta}\cong P M_{N}(\mathcal{A}_{\widetilde{\Theta}} \otimes C(\T^{n-l}))P$. Since $\widetilde{\Theta}$ is non-degenerate, $\mathcal{A}_{\widetilde{\Theta}}$ is a simple $A\mathbb{T}$-algebra with real rank zero by Theorem 3.8 of \cite{Phillips-06}, we may write 
$$\mathcal{A}_{\widetilde{\Theta}}=\lim_{m\rightarrow \infty} \bigoplus _{1 \leq i \leq k_m} M_{s_{m, i}}(C(\T)), \quad k_m, s_{m, i} 
\text{\,\, are integers}.$$
Notice that $$P\in M_{N}(\mathcal{A}_{\widetilde{\Theta}} \otimes C(\T^{n-l}))=\lim_{m\rightarrow \infty} \bigoplus _{1 \leq i \leq k_m}M_{N s_{m, i}}(C(\T)\otimes C(\T^{n-l}) ),$$ we can find projections $P_{m,i}\in M_{N s_{m, i}}(C(\T)\otimes C(\T^{n-l}) )$ for sufficiently large $m$ (we can take subsequences on $m$ here)
such that
$$\mathcal{A}_{\Theta}=\lim_{m\rightarrow \infty} \bigoplus _{1 \leq i \leq k_m}P_{m,i} M_{s_{m, i}}(C(\T)\otimes C(\T^{n-l}))P_{m,i}, \quad k_m, s_{m, i}
\text{\,\, are integers}.$$
Let
\[
C_m = \bigoplus _{1 \leq i \leq k_m} P_{m, i} M_{m, i}(C(\T)\otimes C(\T^{n-1}))P_{m, i}, \quad k_m, s_{m, i} 
\text{\,\, are integers},
\]
then
 $\mathcal{A}_{\Theta} = \underrightarrow{\lim} \,C_m$. 
The K-theory of a unital homogeneous algebra is isomorphic to the K-theory of it's central subalgebra by
Theorem 1.2 of \cite{HM}, so $K$-groups of $P_{m, i} M_{m, i}(C(\T)\otimes C(\T^{n-1}))P_{m, i}$ are finitely generated free abelian groups, thus the $K$-groups of $C_m$ are finitely generated free abelian groups.
\end{proof}

\begin{proposition} \label{P_sameK0}
(Lemma 2.13 of \cite{Elliott-1984}) Let $\Theta$ be a $n \times n$ real skew-symmetric matrix. Then all tracial states $\tau$ on $\mathcal{A}_\Theta$
induce the same map on $K_0(\mathcal{A}_\Theta)$.
\end{proposition}

\begin{proposition}\label{P_AlmostHExist}
Let $\Theta  = (\theta_{j, k})_{ n \times n}\in \mathcal{T}_n$ 
be a $n \times n$ skew-symmetric matrix.
Let $u_1, u_2, \dots, u_n$ 
be the canonical generators of $\mathcal{A}_\Theta$.
Then for any finite subset $\mathcal{G} \subset \mathcal{A}_\Theta$, any $\eta > 0$ and 
any $\ep > 0$, there is a $\delta > 0$ such that
for any unital $C^*$-algebra $A$, 
any $n$-tuple of unitaries $v_1, v_2, \dots, v_n$ in $A$ satisfying
\[
\|v_kv_j - e^{2\pi i \theta_{j, k}}v_jv_k\| < \delta,\quad j, k = 1, 2, \dots, n,
\]
there is a unital $\mathcal{G}$-$\eta$ multiplicative 
c.p.c (completely positive contractive) map 
$L \colon \mathcal{A}_\Theta \rightarrow A$ such that
$\|L(u_i) - v_i\| < \ep$.
\end{proposition}
\begin{proof}
Assume that the proposition is false. 
Let $(\delta_m)_{m=1}^{\infty}$ be a sequence of positive numbers
decreasing to $0$.
Then there is a finite subset $\mathcal{G} \subset \mathcal{A}_\Theta$,
some $\ep, \eta > 0$ such that for any $m$,
there is a unital $C^*$-algebra $A_m$, an $n$-tuple of 
of unitaries  $v_1^{(m)}, v_2^{(m)}, \dots, v_n^{(m)}$ in $A_m$ satisfying
\[
\|v_k^{(m)}v_j^{(m)} - e^{2\pi i \theta_{j, k}}v_j^{(m)}v_k^{(m)}\| < \delta_m,
\quad j, k = 1, 2, \dots, n,
\]
but for any unital $\mathcal{G}$-$\eta$-multiplicative 
c.p.c  map 
$\phi_m \colon \mathcal{A}_\Theta \rightarrow A_m$, we have
$\|\phi_m(u_i) - v_i^{(m)}\| \geq \ep$.

Set $B = \prod_{m=1}^{\infty} A_m / \oplus_{m=1}^{\infty} A_m$.
Let $\pi \colon \prod_{m=1}^{\infty} A_m \rightarrow B$ 
be the canonical quotient map.
Let $v_j = (v_j^{(m)}) \in \prod_{m=1}^{\infty} A_m$.
Then $\{\pi(v_j)\}$ are unitaries satisfying the rotation relation 
with respect to $\Theta$, i.e. $\pi(v_k)\pi(v_j)=e^{2\pi i\theta_{j,k}}\pi(v_j)\pi(v_k).$  
Therefore there is a unital homomorphism 
$\phi \colon \mathcal{A}_\Theta \rightarrow B$. 
By the Choi-Effros lifting theorem, we can lift $\phi$ to a
unital c.p.c map
\[
\tilde{\phi} = (\phi_1, \phi_2, \dots,\phi_m,\dots) 
\colon \mathcal{A}_\Theta \rightarrow \prod_{m=1}^{\infty} A_m.
\]
In particular, each coordinate map $\phi_m$ is unital completely positive,
and we can also assume that they are contractive by normalization. 
By choosing $m$ large enough, we can make sure that $\phi_m$
are $\mathcal{G}$-$\eta$-multiplicative. From our construction, 
\[
\lim_{m \rightarrow \infty} \|\phi_m(u_i) - v_i^{(m)}\| = 0.
\]
This is a contradiction.
\end{proof}

The following follows from functional calculus
and the fact that norm close projections (unitaries) are equivalent:
\begin{lemma}\label{L_AlmostHK0}
Let $\Theta  = (\theta_{j, k})_{ n \times n}\in \mathcal{T}_n$ 
be a $n \times n$ skew-symmetric matrix.
Let $u_1, u_2, \dots, u_n$ be the canonical generators of $\mathcal{A}_\Theta$.
For $k=1,2,\dots,n$ and $j=1,2,\dots,n,$ let $b_{u_k, u_j}$ be the elements of $K_0(\mathcal{A}_\Theta)$ as defined
in Definition \ref{dbotA}. Then 
there exists a finite subset $\mathcal{G} \subset \mathcal{A}_\Theta$, some $\eta > 0$ and $\ep>0$, such that for any unital C*-algebra $A$, 
any $n$-tuple of unitaries $v_1, v_2, \dots, v_n$ in $A$, if $L \colon \mathcal{A}_\Theta \rightarrow A$ 
is a $\mathcal{G}$-$\eta$-multiplicative 
c.p.c. map such that $\|L(u_i) - v_i\| < \ep$,
then
\[
[L](b_{u_k, u_j}) = b_{v_k, v_j}^{\theta_{j,k}}, \quad j , k = 1, 2, \dots, n, j <k.
\]
and
\[
[L]([u_j]) = [v_j], \quad j = 1, 2, \dots n.
\]
\end{lemma}

\begin{lemma}[Lemma 4.1 of \cite{LnTAM}]\label{LsL1}
Let $A$ be a separable unital \CA. For any $\ep>0$ and any finite subset ${\mathcal F}\subset A_{s.a.}$, there exists $\eta>0$ and
a finite subset ${\mathcal G}\subset A_{s.a}$ satisfying the following:
For any ${\mathcal G}$-$\eta$-multiplicative \morp\, $L: A\to B,$ for any
unital \CA\, $B$ with $T(B)\not=\emptyset,$ and any tracial state $t\in T(B),$ there exists a $\tau\in T(A)$ such that
\begin{eqnarray}\label{LsL1-1}
||t\circ L(a)-\tau(a)\|<\ep\,\,\,for \,\,\,all \,\,\, a\in {\mathcal F}.\nonumber
\end{eqnarray}
\end{lemma} 

\begin{lemma} \label{L_distsame} Let $\Theta\in \mathcal{T}_3$ and let $\pi_{\Theta} \colon U_{\infty}(\mathcal{A}_{\Theta}) / CU_{\infty}(\mathcal{A}_{\Theta}) \rightarrow K_1(\mathcal{A}_{\Theta})$
 be  the natural surjective homomorphism. 
Let $[u_1], [u_2], [u_3]$ and $[u_0]$ be the set of generators of $K_1(\mathcal{A}_{\Theta})$
as described in Proposition \ref{P_KofNCT} (with $u_0$ instead of $u$).
Let $J_{\Theta} \colon K_1(\mathcal{A}_{\Theta}) \rightarrow U_{\infty}(\mathcal{A}_{\Theta}) / CU_{\infty}(\mathcal{A}_{\Theta})$
be the embedding induced by
\[
[u_j] \rightarrow \overline{u}_j, \quad j = 0, 1, 2, 3.
\]
Let $[w_1], [w_2], \dots, [w_n]$ be a finite subset of $J_{\Theta}(K_1(\mathcal{A}_{\Theta}))$. 
Then for every $\ep > 0$, there is a finite subset $\mathcal {G} \in \mathcal{A}_{\Theta}$ and a $\eta > 0$
such that,
for any unital $C^*$-algebra $A$ and any unital $\mathcal{G}$-$\eta$-multiplicative c.p.c. map
$L \colon \mathcal{A}_{\Theta} \rightarrow A$, the homomorphism
\[
\alpha_1 \colon J_{\Theta}(K_1(\mathcal{A}_{\Theta}))
\rightarrow U_{\infty}(A) / CU_{\infty}(A), \quad \alpha_1(\overline{u}_j) = \langle L(u_j) \rangle, \,\, j = 0, 1, 2, 3
\]
satisfies $\mathrm{dist}(\overline{ \langle L(w_j) \rangle }, \alpha_1([w_j])) < \ep$, for $j = 1, 2, \dots, n$.
\end{lemma}

\begin{proof}
Let $\ep > 0$ be given. Choose $\delta > 0$ according to Lemma \ref{L_Dist}.

Let $\mathcal{F} = \{u_j, u_j^{-1} \,\vert\, j = 0, 1, 2, 3\}$. For each $k \in \N$, define
\[
\mathcal{F}_k = \{ \mathrm{diag} \{a_1, a_2, \dots, a_k\} \,\vert\, a_j \in \mathcal{F}, \,\,\text{for}\,\, j = 1, 2, \dots, k\}.
\]

It is easy to see that, there is some $N \in \N$ large enough such that, for $j = 1, 2, \dots, n$, 
we can find an element $ b_j \in \mathcal{F}_N$ and unitaries $v_j, s_j$ such that
\[
\| w_j - v_js_jv_j^*s_j^*b_j \| < \frac{\delta}{4} \quad \text{and} \quad \alpha_1(w_j) = \alpha_1(b_j).
\]

Since $\langle L(\cdot) \rangle$ is defined by continuous functional calculus
on the set of unitaries, there is a $\eta_1$ such that whenever 
$L \colon \mathcal{A}_{\Theta} \rightarrow A$ is a unital $\mathcal{F}$-$\eta_1$ multiplicative c.p.c. map,
we have $\| \langle L(u_j) \rangle ^{-1} - \langle L(u_j^{-1}) \rangle \| < \frac{\delta}{4}$, 
for $j = 0, 1, 2, 3$. 

Let $\mathcal{G}_1 = \{v_j, s_j, v_j^*, s_j^* \,\vert\, j = 1, 2, \dots, n\}$.
There is a $\eta_2$ such that whenever 
$L \colon \mathcal{A}_{\Theta} \rightarrow A$ is a unital $\mathcal{G}_1$-$\eta_2$ multiplicative c.p.c. map,
we have 
\[
\| \langle L(v_js_jv_j^*s_j^*) \rangle - 
\langle L(v_j) \rangle \langle L(s_j) \rangle \langle L(v_j) \rangle^* \langle L(s_j)\rangle^*\| < \frac{\delta}{4}, \quad
\text{for}\,\, j = 1, 2, \dots, n.
\]

Let $\mathcal{G} = \mathcal{F} \cup \mathcal{G}_1$ and let $\eta = \min \{ \eta_1, \eta_2, \frac{\delta}{4}\}$. Let $A$ be a unital $C^*$-algebra.
Suppose that $L \colon \mathcal{A}_{\Theta} \rightarrow A$ is a unital $\mathcal{G}$-$\eta$ multiplicative c.p.c. map.
Define $\alpha_1$ to be the homomorphism:
\[
\alpha_1 \colon J_{\Theta}(K_1(\mathcal{A}_{\Theta}))
\rightarrow U_{\infty}(A) / CU_{\infty}(A), \quad \alpha_1(\overline{u}_j) = \langle L(u_j)\rangle, \,\, j = 0, 1, 2, 3.
\]

By our construction, for any $b_j \in \mathcal{F}_N$, there is a unitary $c_j \in U_{\infty}(A)$
such that
\[
\alpha_1(b_j) = \bar{c}_j, \quad \text{and}\quad \|c_j - \langle L(b_j) \rangle\| < \frac{\delta}{4}. 
\]

Therefore we have $\alpha_1(w_j) = \bar{c}_j$. Moreover, we can compute that
\begin{align*}
& \|c_j -  \langle L(w_j) \rangle \langle L(v_j) \rangle  \langle L(s_j) \rangle \langle L(v_j) \rangle^* \langle L(s_j) \rangle^*\| \\
& <\|c_j -  \langle L(b_jv_js_jv_j^*s_j^*) \rangle \langle L(v_j) \rangle  \langle L(s_j) \rangle \langle L(v_j) \rangle^* \langle L(s_j) \rangle^* \| + \frac{\delta}{4} \\
& < \eta + \frac{\delta}{4} + \frac{\delta}{4} + \|c_j - \langle L(b_j) \rangle \| < \delta.
\end{align*}

By Lemma \ref{L_Dist}, we have $\mathrm{dist}(\alpha_1(w_j), \langle L(w_j) \rangle) < \ep$, for  $j=1,\dots,n$.
\end{proof}

\begin{lemma} \label{L_Delta}
Let $(X, d)$ be an infinite compact metric space. Let $C = PM_n(C(X))P$, where $P$ is a projection in $M_n(C(X))$. 
Let $\tau_0$ be a faithful tracial state on $C$. 
Define a function $\Delta \colon (0, 10) \rightarrow (0, 1)$ by
\[
\Delta(r) = \min \{ \mu_{\tau_0}(O_{\frac{r}{10}}(x)) \,\vert\,  x \in X\}, 
\quad 0 < r < 10.
\]
Then $\Delta$ is a strictly increasing function such that $\lim_{r \rightarrow 0} \Delta(r) = 0$.

Moreover, for any $\eta > 0$, there exists a finite subset $\mathcal{H} \in C_+ \backslash \{ 0 \}$ 
and a $\delta > 0$ such that, whenever $L \colon C \rightarrow A$ is a c.p.c. map 
and $\tau$ is a tracial state on $A$, if
\[
|\tau \circ L (h) - \tau_0(h)| < \delta, \quad \text{for all \,\,} h \in \mathcal{H},
\]
then $\mu_{\tau \circ L}(O_r) > \Delta(r)$, for all $r \geq \eta$.
\end{lemma}

\begin{proof}
That $\Delta$ is strictly increasing follows from the fact that $\tau_0$
is faithful. For any $\ep > 0$, 
let $n$ be a natural number such that $\frac{1}{n} < \ep$.
Choose $n$ distinct points $x_1, x_2, \dots, x_n$ in $X$. 
Then there is a $r > 0$ such that the sets $O_{\frac{r}{10}}(x_i)$ are disjoint. 
It follows that $\Delta(r) \leq \frac{1}{n} < \ep$. Therefore $\lim_{r \rightarrow 0} \Delta(r) = 0$.

Now let $\eta > 0$ be given.  
For each integer $i$ such that $0 < i < \frac{4}{\eta}$,
 let $ r_i = i\frac{\eta}{4}$. Choose a finite $r_i$-net $\{ x_{i, j} \vert 1 \leq j \leq k_i\}$ in $X$
(i.e., for any $x$ in $X$, there is some $x_{i, j}$ such that $d(x, x_{i, j}) < r_i$).
For each $i, j$, find a function $g_{i, j}$ which is $1$ on $O_{r_{i-1}}(x_{i, j})$,
$0$ outside $O_{r_i}(x_{i, j})$, and $0 \leq g_{i, j} \leq 1$.
Set $\mathcal{H} = \{g_{i, j} \,\vert\, 0 < i < \frac{4}{\eta}, 1 \leq j \leq k_i\}$.
Let
\[
\delta = \min \{ \Delta(10r_i) - \Delta(9r_i) \,\vert\, 0 < i < \frac{4}{\eta}\}.
\]

Now let $L \colon C \rightarrow A$ be a c.p.c. map and $\tau$ is a tracial state on $A$. Suppose that
\[
|\tau \circ L (h) - \tau_0(h)| < \delta, \quad \text{for all \,\,} h \in \mathcal{H}.
\]
Let $r \geq \eta$. Then there is some $i \geq 4$ such that $ |r - r_i| < \frac{\eta}{4}$.
Let $k = [\frac{i-1}{2}]$. For any open set $O_r$, there is some $j$ 
such that $O_r \supset O_{r_{k}}(x_{k, j})$.

We can then compute that
\begin{align*}
\mu_{\tau \circ L}(O_r) & \geq \mu_{\tau \circ L}(O_{r_{k-1}}(x_{k, j})) > \tau \circ L(g_{k, j})\\
&  \geq \tau_0(h) - \delta  > \tau_0(O_{r_{k-1}}(x_{k, j})) -\delta\\
& \geq \Delta(10r_{k-1}) -\delta \geq \Delta(9r_{k-1}) \\
& = \Delta(r_{9k-9}) > \Delta(r_{2k+4}) > \Delta(r_{i+1})> \Delta(r).\\
\end{align*}
\end{proof}

\begin{lemma} \label{L_tracehom}
Let $\Theta\in \mathcal{T}_n$. 
Let $\tau_0$ be a tracial state on $\mathcal{A}_\Theta$.
Then for any finite subset $\mathcal{H} \subset (\mathcal{A}_\Theta)_+$ and
any $\ep > 0$, there is a finite subset $\mathcal{G} \subset \mathcal{A}_\Theta$,
a $\delta > 0$ and a positive integer $N$ such that, whenever
$A$ is a unital $C^*$-algebra, $\tau$ is a tracial state on $A$,
$L \colon \mathcal{A}_\Theta \rightarrow A$ is a $\mathcal{G}$-$\delta$-multiplicative
c.p.c. map and  $v_1,v_2,\dots,v_n$ are unitaries in $A$ satisfying
\[
\|L(u_j) - v_j \| < \delta \quad \text{and}  \quad
| \tau(v_1^{l_1}v_2^{l_2}\cdots v_n^{l_n}) - \tau_0(u_1^{l_1}u_2^{l_2}\cdots u_n^{l_n})| < \delta
\]
for $j = 1, 2, \dots, n$ and $|l_j| \leq N$,  we have
\[
|\tau \circ L (a) - \tau_0(a)| < \ep, \quad \text{\,\, for all \,\,} a \in \mathcal{H}.
\]
\end{lemma}
\begin{proof}Set $\mathcal{H}=\{a_1,a_2,\dots,a_{m}\},$ we can find $$c_j=\sum\limits_{k_1=-N_{j,1}}^{N_{j,1}}\sum\limits_{k_2=-N_{j,2}}^{N_{j,2}}\sum\limits_{k_3=-N_{j,3}}^{N_{j,3}}\cdots \sum\limits_{k_n=-N_{j,n}}^{N_{j,n}}s_{j,k_1,k_2,k_3,\cdots,k_n}u_1^{k_1} u_{2}^{k_2} u_3^{k_3}\cdots u_n^{k_n}\,\,{\rm for}\,\,j=1,2,\dots,m\,\,$$  ${\rm and\,\,\, some}\,\, N_{j,1},N_{j,2},N_{j,3}\in \mathbb{N}$, $ s_{j,k_1,k_2,k_3}\in \mathbb{C}$ such that $\|a_j-c_j\|<\frac{\ep}{8}$ for $j=1,2,\dots,m.$

Set $N=\sup\limits_{j=1,2,\dots,m}\{N_{j,1},N_{j,2},N_{j,3},\cdots, N_{j,n}\}$. Choose small enough $\delta$ such that

\begin{align*}
&\phantom{aa} |\tau\circ L(a_j)-\tau_0(a_j)|&\\
& \leq |\tau\circ L(c_j)-\tau_0(c_j)|+\frac{\ep}{4}&\\
 &\leq \sum\limits_{k_1=-N_{j,1}}^{N_{j,1}}\sum\limits_{k_2=-N_{j,2}}^{N_{j,2}}\sum\limits_{k_3=-N_{j,3}}^{N_{j,3}}\cdots \sum\limits_{k_n=-N_{j,n}}^{N_{j,n}}|s_{j,k_1,k_2,k_3,\cdots,k_n}(\tau(L(u_1^{k_1} u_{2}^{k_2} u_3^{k_3}\cdots u_n^{k_n}))&\\
 &-\tau_0(u_1^{k_1} u_{2}^{k_2} u_3^{k_3}\cdots u_n^{k_n}))|+\frac{\ep}{4}&\\
 &\leq \sum\limits_{k_1=-N_{j,1}}^{N_{j,1}}\sum\limits_{k_2=-N_{j,2}}^{N_{j,2}}\sum\limits_{k_3=-N_{j,3}}^{N_{j,3}}\cdots \sum\limits_{k_n=-N_{j,n}}^{N_{j,n}}|s_{j,k_1,k_2,k_3,\cdots,k_n}(\tau(L(u_1)^{k_1} L(u_{2})^{k_2} L(u_3)^{k_3}\cdots L(u_n)^{k_n})&\\
& -\tau_0(u_1^{k_1} u_{2}^{k_2} u_3^{k_3}\cdots u_n^{k_n}))|+\frac{2\ep}{4}&\\ 
  &\leq \sum\limits_{k_1=-N_{j,1}}^{N_{j,1}}\sum\limits_{k_2=-N_{j,2}}^{N_{j,2}}\sum\limits_{k_3=-N_{j,3}}^{N_{j,3}}\cdots \sum\limits_{k_n=-N_{j,n}}^{N_{j,n}}|s_{j,k_1,k_2,k_3,\cdots,k_n}(\tau(v_1^{k_1} v_{2}^{k_2} v_3^{k_3}\cdots v_n^{k_n}) &\\
  &-\tau_0(u_1^{k_1} u_{2}^{k_2} u_3^{k_3}\cdots u_n^{k_n}))|+\frac{3\ep}{4}&\\
&\leq \ep.&
\end{align*}

\end{proof}

\begin{lemma}\label{L_Commongap}
Let $\Theta=(\theta_{j,k})\in \mathcal{T}_n$. There exist $\delta>0$ and $\theta\in[0,1)$
such that for any unital $C^*$-algebra $A$, any unitaries $v_1, v_2, \dots, v_n$ in $A$ such that 
\[
\|v_kv_j-e^{2\pi i \theta_{j,k}}v_jv_k\|<\delta, 
\text{\,\,for\,\,} j, k = 1, 2, \dots, n,
\]
and the spectrum 
${\rm sp}(v_kv_jv_k^*v_j^*)$ has a gap at $e^{ (2\pi \theta+\pi)i }$,
for all $j, k$. In particular, $\log_{\theta}(v_kv_jv_k^*v_j^*)$ is well defined and $\log_{\theta}(v_kv_jv_k^*v_j^*)=\log_{\theta_{j,k}}(v_kv_jv_k^*v_j^*)$ for $j,k=1,2,\dots,n.$
\end{lemma}
\begin{proof}
Choose $\delta > 0$ small enough so that the set
\[
\cup_{j, k = 1, 2, ,\dots, n}\{e^{2\pi i t} \,\vert\, 
\theta_{j, k} -\delta \leq t \leq \theta_{j, k} + \delta\} 
\]
is not the whole circle. Choose 
$\theta \in [0, 1)$ so that $e^{(2\pi \theta+\pi )i }$ is not in the above set.
If $\|v_kv_j-e^{2\pi i \theta_{j,k}}v_jv_k\|<\delta$, then 
${\rm sp}(v_kv_jv_k^*v_j^*) \subset \{e^{2\pi i t} \,\vert\, \theta_{j, k} -\delta \leq t \leq \theta_{j, k} + \delta\}$.
Thus $e^{(2\pi \theta+\pi )i }$ is a common gap of the spectra.
\end{proof}

We are now ready to prove the main theorem of this section.

\begin{theorem}\label{T_stabilityTR1}
Let $\Theta = (\theta_{j, k})_{3 \times 3} \in \mathcal{T}_3$.
Let $u_1, u_2, u_3$ be the canonical unitaries in $\mathcal{A}_\Theta$
and $\tau_0$ be a faithful tracial state on $\mathcal{A}_\Theta$.
For any $\varepsilon>0$, there exists $0<\delta<\frac{1}{2}$ 
 and a positive integer $N > 0$ satisfying the following:
For any unital simple separable $C^*$-algebra $A$ with tracial rank at most one, any three unitaries $v_1,v_2,v_3\in A$ such that
\begin{flalign}
& {\rm(1)} \,\,
\|v_k v_j-e^{2\pi i \theta_{j,k}}v_j v_k\|<\delta,\label{ml1}&\\
& {\rm(2)} \,\,
\frac{1}{2\pi i}\tau(\log_{\theta}(v_k v_j v_k^* v_j^*))=\theta_{j,k},\, j,k = 1,2,3,\, and\,\theta\,  is\, constructed\, as\, in\, Lemma\, \ref{L_Commongap}, \label{ml2}&\\
&{\rm (3)}\,\,
|\tau(v_1^{l_1}v_2^{l_2}v_3^{l_3}) - \tau_0(u_1^{l_1}u_2^{l_2}u_3^{l_3})|<\delta,
\quad \text{for all}\,\, \tau \in T(A) \text{\,\, and \,\,}
 |l_1|, |l_2|, |l_3| \leq N,\label{ml211}&
\end{flalign}
 there exists a triple of unitaries $\tilde{v}_1,\tilde{v}_2,\tilde{v}_3\in A$ such that
$$\tilde{v}_k\tilde{v}_j=e^{2\pi i\theta_{j,k} }\tilde{v}_j\tilde{v}_k\,\,\,\,\mbox{and}\,\,\,\,\|\tilde{v}_j-v_j\|<\varepsilon,\,\,\,j,k=1,2,3.$$
\end{theorem}

\begin{proof}
Let $\varepsilon>0$ be given. Set $\mathcal{F}=\{1_{\mathcal{A}_{\Theta}},u_1,u_2,u_3\}$.
By Proposition \ref{P_RAisAH}, 
there are $C_n $ with finitely generated free abelian $K$-groups which are finite direct sum  of homogeneous $C^*$-algebras such that
$\mathcal{A}_\Theta = \overline{\bigcup_{n=1}^{\infty} C_n}$.
Choose $n$ large enough and a finite subset $\widetilde{\mathcal{F}} \subset C_n := C$
such that for any $a$ in $\mathcal{F}$,
there is a $\tilde{a} \in \widetilde{\mathcal{F}}$ such that $\| a - \tilde{a}\| < \frac{\ep}{3}$.

By Proposition \ref{P_KofNCT}, $K_1(\mathcal{A}_\Theta)$  is isomorphic to $\Z^4$ 
which is generated by $[u_1], [u_2], [u_3]$
and $[u_0]$ for some $ u_0 \in U_{\infty}(\mathcal{A}_\Theta)$. We define a splitting map:
$$
J_\Theta \colon K_1(\mathcal{A}_\Theta) \rightarrow U_{\infty} (\mathcal{A}_\Theta) / CU_{\infty}(\mathcal{A}_{\Theta}),
\quad J_\Theta([u_i]) = \overline{u}_i, \,\, 0 \leq i \leq 3.$$

For simplicity, we assume that $C=PM_n(C(X))P$. The proof for $C$ being the direct sum
of such algebras is exactly the same.
Let $\iota \colon C \rightarrow \mathcal{A}_{\Theta}$ be the inclusion map.
Since $K_1(C)$ is finitely generated free abelian group, there is a splitting map
$J_C \colon K_1(C) \rightarrow U_{\infty} (C) / CU_{\infty}(C)$ such that the following diagram 
commutes:
\begin{center}
\begin{tikzcd}
U_{\infty} (C) / CU_{\infty}(C)  \arrow[d, "\iota^{\ddagger}"] & K_1(C) \arrow[l, "J_C"'] \arrow[d, "\iota_*"] \\
U_{\infty} (\mathcal{A}_\Theta) / CU_{\infty}(\mathcal{A}_{\Theta})   & K_1(\mathcal{A}_\Theta)
\arrow[l, "J_{\mathcal{A}_\Theta}"'].
\end{tikzcd}
\end{center}
Define a function $\Delta \colon (0, 10) \rightarrow (0, 1)$ by
\[
\Delta(r) = \min \{ \mu_{\tau_0}(O_{\frac{r}{10}}(x)) \,\vert\,  x \in X\}, 
\quad 0 < r < 10.
\]

Choose $\eta> 0$, $\delta_0 > 0$ (in place of $\delta$), 
a finite subset $\widetilde{\mathcal{G}} \subset C$,
a finite subset $\widetilde{\mathcal{P}} \subset \underline{K}(C)$, 
a finite subset $\widetilde{\mathcal{H}} \subset C_{s.a.}$ and 
a finite subset $\widetilde{\mathcal{U}} \in U_c(K_1(C))$
with respect to $\Delta$, $\ep/3$ (in place of $\ep$) and $\widetilde{\mathcal{F}}$ (in place of $\mathcal{F}$) according to Theorem \ref{T_unique}.

For the matter of convenience, the trace $\tau_0$ restricted to $C$ is denoted by the same symbol.
By Lemma \ref{L_Delta}, there exists a finite subset $\widetilde{\mathcal{H}}_1 \in C_+ \backslash \{0\}$
and a $\delta_1 \leq \delta_0$ such that, for any $C^*$-algebra $A$,
whenever $L \colon C \rightarrow A$ is a c.p.c. map and $\tau$ is a tracial state on $A$, if
\[
|\tau \circ L (h) - \tau_0(h)| < \delta_1, \quad \text{for all \,\,} h \in \widetilde{\mathcal{H}}_1,
\]
then $\mu_{\tau \circ L}(O_r) > \Delta(r)$, for all $r \geq \eta$.

Let $\mathcal{H} = \iota(\widetilde{\mathcal{H}} \cup \widetilde{\mathcal{H}}_1)$. Choose a finite subset $\mathcal{G}_0\subset \mathcal{A}_{\Theta}$, a $\delta_2 > 0$ and a positive integer $N$
with respect to $\mathcal{H}$ and $\delta_1$ (in place of $\ep$) according to Lemma \ref{L_tracehom}.

Let $\overline{\mathcal{U}} = \{ \overline{\iota(w)} \,\vert\, w \in \widetilde{\mathcal{U}} \}$.
Then $\overline{\mathcal{U}}$ is a finite subset of $J_{\Theta}(K_1(\mathcal{A}_{\Theta}))$.
By Lemma \ref{L_distsame}, there is a finite subset $\mathcal{G}_1$ and a $\delta_3$ such that, 
for any unital $C^*$-algebra $A$ and any unital $\mathcal{G}_1$-$\delta_3$-multiplicative c.p.c. map
$L \colon \mathcal{A}_{\Theta} \rightarrow A$, there is a homomorphism
\[
\alpha_1 \colon J_{\Theta}(K_1(\mathcal{A}_{\Theta})) \rightarrow U_{\infty}(A)/CU_{\infty}(A)
\]
such that $\mathrm{dist}(\overline{\langle L(\iota(w)) \rangle}, \alpha_1(\overline{\iota(w)}) < \delta_0$,
for all $w \in \widetilde{\mathcal{U}}$.

Let $\mathcal{P} = [\iota](\widetilde{\mathcal{P}})$.
Since $K_*(\mathcal{A}_{\Theta})$ is torsion free, 
 by Proposition 2.4 of \cite{Schochet}, we have 
$K_*(\mathcal{A}_{\Theta}; \Z/n\Z) \cong K_*(\mathcal{A}_{\Theta}) \otimes \Z/n\Z$.
By the description of $K_*(\mathcal{A}_{\Theta})$ (See Proposition \ref{P_KofNCT}),
we may assume without loss of generality that 
\[
\mathcal{P}=\{[1_{\mathcal{A}_{\Theta}}],b_{u_2,u_1},b_{u_3,u_1},b_{u_3,u_2},[u_1],[u_2],[u_3],[u_0]\}.
\]

By Lemma \ref{L_AlmostHK0}, we can choose finite subset $\mathcal{G}_2 \subset \mathcal{A}_\Theta$,
$0 < \ep_0 < \frac{\ep}{6}$, and $\delta_4 > 0$ so that whenever
$v_1, v_2, v_3$ are unitaries in $A$ and
$L \colon \mathcal{A}_\Theta \rightarrow A$ 
is a $\mathcal{G}_2$-$\delta_4$-multiplicative 
c.p.c. map such that $\|L(u_i) - v_i\| < \ep_0$,
then
\[
[L](b_{u_k, u_j}) = b_{v_k, v_j}^{\theta_{j,k}}, \quad j , k = 1, 2, 3, j <k.
\]
\[
[L]([u_j]) = [v_j], \quad j = 1,2, 3.
\]
Finally, we choose $\theta$ and $\delta_5$ as in Lemma \ref{L_Commongap} for $\Theta$.

Set $\mathcal{G} = \iota(\widetilde{\mathcal{G}}) \cup \mathcal{G}_0 \cup \mathcal{G}_1 \cup \mathcal{G}_2$. 
Let $\delta_6= \min\{\delta_0, \delta_1, \delta_2, \delta_3, \delta_4,\delta_5\}$.
Find a positive number $\delta \leq \delta_6$ according to $\mathcal{G}$, $\delta_6 $ (in place of $\eta$) and $\ep_0 > 0$ ( in place of $\ep$) as in Proposition \ref{P_AlmostHExist}. 

Now suppose that $A$ is a unital simple $C^*$-algebra with tracial rank at most one.
Let $v_1, v_2, v_3\in A$ be unitaries such that
\begin{flalign}
& {\rm(1)} \,\,
\|v_kv_j-e^{2\pi i \theta_{j,k}}v_jv_k\|<\delta,& \\
& {\rm(2)} \,\,
\frac{1}{2\pi i}\tau(\log_{\theta}(v_kv_jv_k^*v_j^*))=\theta_{j,k}\quad j,k = 1,2,3\,\,{\rm and}& \\
&{\rm (3)}\,\,
|\tau(v_1^{l_1}v_2^{l_2}v_3^{l_3}) - \tau_0(u_1^{l_1}u_2^{l_2}u_3^{l_3})|<\delta,
\quad \text{for all}\,\, \tau \in T(A) \text{\,\, and \,\,}
 |l_1|, |l_2|, |l_3| \leq N.&             
\end{flalign}

Define $\kappa_0 \colon K_0(\mathcal{A}_{\Theta}) \rightarrow K_0(A)$ to be the homomorphism induced
by 
\[\kappa_0([1])=[1_A]\quad {\rm and}\quad
\kappa_0(b_{u_k, u_j}) = b_{v_k, v_j}^{\theta_{j, k}} = [L](b_{u_k, u_j}), \quad j, k = 1, 2, 3.
\]
We claim that this is a positive homomorphism. Indeed, let $[p] \in K_0(\mathcal{A}_{\Theta})$
be a positive element. Then $\tau_{\Theta}(p) > 0$. There are integers $n_1, n_2, n_3$ and $n_4$
such that $[p] = n_1 [1_{\mathcal{A}_{\Theta}}] + n_2b_{u_2, u_1} + n_3b_{u_3, u_1} + n_4b_{u_3, u_2}$.
So for any tracial state $\tau$ on $A$, by the generalized Exel trace formula (see Theorem \ref{trace formula}) and our assumption,
we can compute that
\begin{align*}
\kappa_0([p])[\tau] 
&= n_1 \tau(1) + n_2 \tau(b_{v_2, v_1}^{\theta_{1,2}}) + n_3 \tau(b_{v_3, v_1}^{\theta_{1,3}}) + n_4 \tau(b_{v_3, v_2}^{\theta_{2,3}})\\
&= n_1 1 + n_2 \theta_{1, 2} + n_3 \theta_{1, 3} + n_4 \theta_{2, 3}\\
&= n_1 \tau_{\Theta}(1) + n_2 \tau_{\Theta}(b_{u_2, u_1}) + n_3 \tau_{\Theta}(b_{u_3, u_1}) + n_4 \tau_{\Theta}(b_{u_3, u_1}) = \tau_{\Theta}(p) > 0.
\end{align*}
Since $A$ has strict comparison, this shows that $\kappa_0([p])$ is positive. 

Define a map
\[
\gamma: T(A) \rightarrow T_f(\mathcal{A}_{\Theta}) \quad \gamma(\tau) = \tau_0, \,\,\text{for all}\,\,\tau \in T(A).
\]
It induces an affine map $\gamma_{\sharp} \colon \mathrm{Aff}(T(\mathcal{A}_{\Theta})) \rightarrow \mathrm{Aff}(T(A))$. By Proposition \ref{P_sameK0}, for any tracial state
$\tau \in T(A)$,  we can compute that
\begin{align*}
\gamma_{\sharp}(\hat{p}) (\tau) 
&= \hat{p} (\gamma(\tau)) = \hat{p}(\tau_0) \\
&= \tau_0(p) = \tau_{\Theta}(p) \\
&= \kappa_0([p])(\tau).
\end{align*}
It follows that $\gamma_{\sharp}( \overline{\rho_{\mathcal{A}_{\Theta}}(K_0(\mathcal{A}_{\Theta}))})
\subset  \overline{\rho_A(K_0(A))}$.
Therefore there is an induced homomorphism (denoted by the same symbol)
\[
\gamma_{\sharp} \colon \mathrm{Aff}(T(\mathcal{A}_{\Theta})) / \overline{\rho_{\mathcal{A}_{\Theta}}(K_0(\mathcal{A}_{\Theta}))}
\rightarrow \mathrm{Aff}(T(A))/\overline{\rho_A(K_0(A))}.
\]

By Proposition \ref{P_AlmostHExist}, there is a unital $\mathcal{G}$-$\delta_5$-multiplicative c.p.c. map
$L \colon \mathcal{A}_{\Theta} \rightarrow A$ such that $\| L(u_j) - v_j \| < \ep_0$, for $j = 1, 2, 3$.
By Lemma \ref{L_distsame}, the homomorphism
\[
\alpha_1 \colon J_{\Theta}(K_1(\mathcal{A}_{\Theta})) \rightarrow U_{\infty}(A)/CU_{\infty}(A), \quad 
\alpha_1(\overline{u}_j) = \overline{\langle L(u_j) \rangle}, \text{\,\,for\,\,} j = 0, 1, 2, 3
\]
satisfies $\mathrm{dist}(\overline{\langle L(\iota(w_j)) \rangle}, \alpha_1(\overline{\iota(w_j)}) < \delta_0$, for
all $w_j \in \widetilde{U}$.

Since $K_1(\mathcal{A}_{\Theta})$ is free abelian, it is clear that there is a homomorphism
$\kappa_1 \colon K_1(\mathcal{A}_{\Theta}) \rightarrow K_1(A)$ such that the following diagram commutes:
\begin{center}
\begin{tikzcd}
U_{\infty} (\mathcal{A}_\Theta) / CU_{\infty}(\mathcal{A}_{\Theta})\arrow[d, "\alpha_1"]   & K_1(\mathcal{A}_\Theta) \arrow[l, "J_{\mathcal{A}_{\Theta}}"'] \arrow[d, "\kappa_1"]\\
U_{\infty} (A) / CU_{\infty}(A)   & \arrow[l, "J_{A}"']K_1(A) \\
\end{tikzcd}
\end{center}

Use the split exact sequence (\ref{E_Thomsen_es}),  we can find a homomorphism
$\alpha \colon U_{\infty}(\mathcal{A}_{\Theta})/CU_{\infty}(\mathcal{A}_{\Theta}) 
\rightarrow U_{\infty}(A)/CU_{\infty}(A)$ so that the following diagram commutes:

\begin{tikzcd}
0 \arrow[r] & \mathrm{Aff}(T(\mathcal{A}_{\Theta})) / \overline{\rho_{\Theta}(K_0(\mathcal{A}_{\Theta}))}
\arrow[r]\arrow[d, "\gamma_{\sharp}"] & U_{\infty}(\mathcal{A}_{\Theta})/CU_{\infty}(\mathcal{A}_{\Theta})
\arrow[r, "\pi_{\mathcal{A}_{\Theta}}"] \arrow[d, dashed, "\alpha"] & K_1(\mathcal{A}_{\Theta}) \arrow[l, shift left, "J_{\mathcal{A}_{\Theta}}"]\arrow[r]\arrow[d, "\kappa_1"] & 0. \\
0 \arrow[r] &\mathrm{Aff}(T(A)) / \overline{\rho_A(K_0(A))}
\arrow[r] &U_{\infty}(A)/CU_{\infty}(A) 
\arrow[r, "\pi_A"] & K_1(A) \arrow[r] & 0
\end{tikzcd}

Let $\kappa =(\kappa_0, \kappa_1)$. Then $\kappa \in KL_e(\mathcal{A}_{\Theta}, A)^{++}$.
It follows from the commutative diagram that $\kappa, \gamma$ and $\alpha$ are compatible.
Therefore by Theorem \ref{T_exist} there is a unital monomorphism $\phi \colon \mathcal{A}_{\Theta} \rightarrow A$ such that
\[
[\phi] = \kappa,\quad \phi_{\sharp} = \gamma, \quad \text{and\,\,\,\,} \phi^{\ddagger} = \alpha.
\]

We now compare the two maps $\widetilde{L} = L \circ \iota$ and $\tilde{\phi} = \phi \circ \iota$. It's easy to see from our construction that
\[
[\tilde{L}]\,\vert_{\widetilde{\mathcal{P}}} = [\tilde{\phi}]\,\vert_{\widetilde{\mathcal{P}}}
\]
and 
\[
\mathrm{dist}(\overline{\langle \widetilde{L}(w) \rangle }, \overline{\tilde{\phi}(w)})
= \mathrm{dist}(\overline{\langle L(\iota(w)) \rangle }, \overline{\phi (\iota(w))}) < \delta_0,
\quad \text{for all\,\,} w \in \widetilde{\mathcal{U}}.
\]

By Lemma \ref{L_tracehom}, we have
\[
\| \tau \circ L (h) - \tau_0(h) \| < \delta_1, \quad \text{for all}\,\, h \in \mathcal{H}.
\]
Therefore 
\[
\|\tau \circ \widetilde{L} (g) - \tau \circ \tilde{ \phi }(g) \|
= \|\tau \circ L (\iota(g)) - \tau_0(\iota(g)) \|< \delta_1 \leq \delta_0,
\]
for all $\tau \in T(A)$ and $g \in \mathcal{H}_0$.

The same computation shows that
\[
\|\tau \circ \widetilde{L} (h) - \tau_0(\iota(h)) \| < \delta_1, \quad \text{for all\,\,} h \in \widetilde{\mathcal{H}}_1
\]

and 
\[
\|\tau \circ \tilde{\phi} (h) - \tau_0(\iota(h)) \| = 0 < \delta_1, \quad \text{for all\,\,} h \in \widetilde{\mathcal{H}}_1.
\]

So by Lemma \ref{L_Delta}, we have
\[
\mu_{\tau \circ L} (O_r) > \Delta(r), \quad {\rm and}\quad \mu_{\tau \circ \tilde{\phi}}(O_r) > \Delta(r), 
\quad \text{\,\,for all\,\,} r \geq \eta. 
\]

By Theorem \ref{T_unique}, there is a unitary $W$ such that
\[
\| W^*\tilde{\phi}(a)W - \widetilde{L}(a) \| < \frac{\ep}{3}, \quad \text{\,\,for all \,\,} a \in \widetilde{\mathcal{F}}.
\]

For $j = 1, 2, 3$, find $\tilde{u}_j \in \widetilde{\mathcal{F}}$ such that $ \| \iota(\tilde{u}_j) - u_j \| < \frac{\ep}{6}$.
We can then compute that
\begin{align*}
\| \iota(W)^* \phi(u_j)\iota(W) - v_j\| &\leq \|\iota(W)^* \phi(\iota(\tilde{u}_j))\iota(W) - L(u_j)\| + \frac{\ep}{3}\\
& \leq \|\iota(W^*\tilde{\phi}(\tilde{u}_j)W) - L(\tilde{u}_j)\| + \frac{\ep}{3} + \frac{\ep}{3}\\
& \leq \| W^*\tilde{\phi}(\tilde{u}_j)W - \widetilde{L}(\tilde{u}_j)\| < \ep.
\end{align*}

Let $\tilde{v}_j = \iota(W)^*\phi(u_j)\iota(W)$. Then these are unitaries in $A$ satisfies the rotation 
relation with respect to $\Theta$ such that $\|\tilde{v}_j - v_j\| < \ep$.
\end{proof}

If $\Theta$ is non-degenerate, we can get a more satisfactory result. The next lemma shows that
the extra trace condition (\ref{ml211}) in Theorem \ref{T_stabilityTR1} 
is in fact automatic in this case.

\begin{lemma} \label{L_autotrace}
Suppose that $\Theta = (\theta_{j, k})_{3 \times 3}\in \mathcal{T}_3$ be a non-degenerate real skew-symmetric matrix. 
Let $\tau_{\Theta}$ be the unique tracial state on $\mathcal{A}_{\Theta}$.
Then for any $\ep > 0$, any positive integer $N$, there is a $\delta > 0$
such that whenever $A$ is a unital $C^*$-algebra, $v_1, v_2, v_3$ are unitaries in $A$ such that
\[
\|v_kv_j - e^{2\pi i \, \theta_{j, k}}v_jv_k\| < \delta,
\]
we have 
\[
|\tau(v_1^{l_1}v_2^{l_2}v_3^{l_3}) - \tau_{\Theta}(u_1^{l_1}u_2^{l_2}u_3^{l_3})|< \ep,
\quad \text{for all}\,\, \tau \in T(A) \text{\,\, and \,\,}
|l_1|, |l_2|, |l_3| \leq N.
\]
\end{lemma}

\begin{proof}

Since $\mathcal{A}_{\Theta}$ has a unique tracial state in this case,
by Lemma 4.1 of \cite{LnTAM}, there is a finite subset $\mathcal{G}_1 \subset \mathcal{A}_{\Theta}$
and $\eta_1 > 0$ such that, whenever $A$ is a unital $C^*$-algebra 
and $L \colon \mathcal{A}_{\Theta}\rightarrow A$ is a unital $\mathcal{G}_1$-$\eta_1$-multiplicative c.p.c. map,
then
\[
|\tau \circ L(u_1^{l_1}u_2^{l_2}u_3^{l_3}) - \tau_{\Theta}(u_1^{l_1}u_2^{l_2}u_3^{l_3})|<\frac{\ep}{2}，
\quad \text{for all}\,\, \tau \in T(A) \text{\,\, and \,\,}
| l_1|, |l_2|, |l_3| \leq N.
\]

It is also easy to find a finite subset $\mathcal{G}_2 \subset \mathcal{A}_{\Theta}$, a 
$\eta_2 > 0$ and a $\ep_0 > 0$ such that, whenever
 $L \colon \mathcal{A}_{\Theta}$ is a unital $\mathcal{G}_2$-$\eta_2$-multiplicative c.p.c. map
 with $\| L(u_j) - v_j \| < \ep_0$ for $j= 1, 2, 3$,
 we have
 \[
 \|L(u_1^{l_1}u_2^{l_2}u_3^{l_3}) - v_1^{l_1}v_2^{l_2}v_3^{l_3} \| < \frac{\ep}{2}.
 \]
 
 Let $\mathcal{G} = \mathcal{G}_1 \cup \mathcal{G}_2$ and $\eta = \min \{\eta_1, \eta_2\}$.
 Find $\delta > 0$ according to $\mathcal{G}$, $\eta$ and $\ep_0$ (in place of $\ep$) 
 as in Proposition \ref{P_AlmostHExist}. 
 
 If $A$ is a unital $C^*$-algebra, $v_1, v_2, v_3$ are unitaries such that
\[
\|v_kv_j - e^{2\pi i \, \theta_{j, k}}v_jv_k\| < \delta,
\]
by Proposition \ref{P_AlmostHExist}, there is $\mathcal{G}$-$\eta$-multiplicative c.p.c. map
such that $\| L(u_j) - v_j \| < \ep_0$ for $j= 1, 2, 3$.
Then
\[
|\tau(v_1^{l_1}v_2^{l_2}v_3^{l_3}) - \tau_{\Theta}(u_1^{l_1}u_2^{l_2}u_3^{l_3})|
< |\tau(L(u_1^{l_1}u_2^{l_2}u_3^{l_3})) - \tau_{\Theta}(u_1^{l_1}u_2^{l_2}u_3^{l_3})| + \frac{\ep}{2}
<  \frac{\ep}{2} +  \frac{\ep}{2} = \ep,
\]
for all $\tau \in T(A)$ and  $|l_1|, |l_2|, |l_3| \leq N$.
\end{proof}

\begin{corollary}\label{nondegenerate stablity}
Let $\Theta = (\theta_{j, k})_{3 \times 3} \in \mathcal{T}_3$ be a non-degenerate real skew-symmetric matrix. 
Then for any $\varepsilon>0$, there exists $\delta > 0$ satisfying the following:
For any unital simple separable $C^*$-algebra $A$ with tracial rank at most one,
any three unitaries $v_1,v_2,v_3 \in A$ such that
\[
\|v_kv_j-e^{2\pi i\, \theta_{j,k}}v_jv_k\|<\delta,\,\, j, k = 1, 2, 3,
\]
 there exists a triple of unitaries $\tilde{v}_1,\tilde{v}_2,\tilde{v}_3\in A$ such that
\[
\tilde{v}_k\tilde{v}_j=e^{2\pi i \, \theta_{j,k}}\tilde{v}_j\tilde{v}_k\,\,\,\,\mbox{and}\,\,\,\,\|\tilde{v}_j-v_j\|<\varepsilon,\,\,\,j,k=1,2,3
\]
if and only if
\[
\frac{1}{2\pi i}\tau(\log_{\theta}(v_kv_jv_k^*v_j^*))=\theta_{j,k}\,\,{\rm for}\,\, j,k = 1,2,3\,\, {\rm and}\,\, {\rm all}\,\, \tau\in T(A),
\]
where $\theta$ is constructed as in Lemma \ref{L_Commongap}.
\end{corollary}
\begin{proof}
The `if' part follows from Theorem \ref{T_stabilityTR1} and Lemma \ref{L_autotrace}.
The `only if' part follows from Corollary \ref{C_ETF}.
\end{proof}

From Lemma \ref{3-dimensional rank}, we can get the following corollary.
\begin{corollary}
Let $\Theta = (\theta_{j, k})_{3 \times 3} \in \mathcal{T}_3$. Suppose $\dim_{\mathbb{Q}}(\rm{span}_{\mathbb{Q}}(1,\theta_{1,2}
\theta_{1,3},\theta_{2,3}))\geq 3.$
Then for any $\varepsilon>0$, there exists $\delta > 0$  satisfying the following:
For any unital simple separable $C^*$-algebra $A$ with tracial rank at most one,
any three unitaries $v_1,v_2,v_3 \in A$ such that
\[
\|v_kv_j-e^{2\pi i\, \theta_{j,k}}v_jv_k\|<\delta,\,\, j, k = 1, 2, 3,
\]
 there exists a triple of unitaries $\tilde{v}_1,\tilde{v}_2,\tilde{v}_3\in A$ such that
\[
\tilde{v}_k\tilde{v}_j=e^{2\pi i \, \theta_{j,k}}\tilde{v}_j\tilde{v}_k\,\,\,\,\mbox{and}\,\,\,\,\|\tilde{v}_j-v_j\|<\varepsilon,\,\,\,j,k=1,2,3
\]
if and only if
\[
\frac{1}{2\pi i}\tau(\log_{\theta}(v_kv_jv_k^*v_j^*))=\theta_{j,k}\,\,{\rm for}\,\, j,k = 1,2,3\,\, {\rm and}\,\, {\rm all}\,\, \tau\in T(A),  
\]
where $\theta$ is constructed as in Lemma \ref{L_Commongap}.

\end{corollary}

\begin{remark} It is natural to ask what happens if there are $n$ unitary elements $v_1,v_2,\cdots,v_n$ in $A$ for $n\geq 4$ and
$\|v_k v_j-e^{2\pi i \theta_{j,k}}v_j v_k\|<\delta, j,k=1,2,\dots,n$ as in Theorem \ref{T_stabilityTR1} and Corollary \ref{nondegenerate stablity}. 
For simplicity, let us consider the case $n=4,$ let $\Theta=(\theta_{j,k})_{4\times 4},$ we can compute $K_0$-group of higher dimensional noncommutative torus $\mathcal{A}_{\Theta},$ let $u_1,u_2,u_3,u_4$ be the generators of $\mathcal{A}_{\Theta},$ it is well known that
\begin{eqnarray*}
&K_0(\mathcal{A}_{\Theta})\phantom{aaaaaaaaaaaaaaaaaaaaaaaaaaaaaaaaaaaaaaaaaaaaaaaaaaaaaaaaaaaaaaaaa}\\
\cong&\mathbb{Z}^8 \phantom{aaaaaaaaaaaaaaaaaaaaaaaaaaaaaaaaaaaaaaaaaaaaaaaaaaaaaaaaaaaaaaaaaaaaa}\\
\cong &\mathbb{Z}[1]+\mathbb{Z}[b_{u_1,u_2}]+\mathbb{Z}[b_{u_1,u_3}]+\mathbb{Z}[b_{u_1,u_4}]+\mathbb{Z}[b_{u_2,u_3}]+
\mathbb{Z}[b_{u_2,u_4}]+\mathbb{Z}[b_{u_3,u_4}]+\mathbb{Z}[b_{u_1,u_2,u_3,u_4}],
\end{eqnarray*}
where $b_{u_i,u_j}$ is as defined in Definition \ref{D_RieffelProj} and  the high dimensional Bott element $b_{u_1,u_2,u_3,u_4}$ appears (One can refer to  \cite{Chakraborty} for higher dimensional Bott element $b_{u_1,u_2,u_3,u_4}$). $b_{u_1,u_2,u_3,u_4}$ is a new obstruction to stability. So for our proof,  the conditions of our Theorem \ref{T_stabilityTR1} are not sufficient to get the conclusion in the case $n\geq 4.$
\end{remark}

\section{Stability of rotation relation in purely infinite simple $C^*$-algebras}

\begin{definition}
(Definition 2.2 of \cite{LinTAM04}) Let $A$ be a $C^*$-algebra. 
Let $\mathbf{D}$ be a class of $C^*$-algebras. 
We say $A$ is weakly stable with respect to $\mathbf{D}$
if and only if for any $\ep > 0$, any finite subset $\mathcal{F} \subset A$,
there exists a $\delta > 0$ and a finite subset $\mathcal{G} \subset A$
satisfying the following:
for any $B \in \mathbf{D}$,
any $\mathcal{G}$-$\delta$-multiplicative c.p.c. map $L \colon A \rightarrow B$,
there is a homomorphism $\phi \colon A \rightarrow B$ such that
\[
\| \phi(a) - L(a) \| < \ep, \quad \text{for all}\,\, a \in \mathcal{F}.
\] 
\end{definition}

\begin{lemma} \label{L_wstable_matrix}
Let $A$ be a unital $C^*$-algebra and $\mathbf{D}$ be a class of $C^*$-algebras.
Suppose that $A$ is weakly stable with respect to $\mathbf{D}$,
then $M_n(A)$ is weakly stable with respect to $\mathbf{D}$ for $n=1,2,\cdots.$
\end{lemma}

\begin{proof}
Let $\ep > 0$ and finite subset $\mathcal{F} \subset M_n(A) \cong M_n \otimes A$ be given.
Let $e_{i, j}$, $1 \leq i, j \leq n$ be the matrix units of $M_n$.
Without loss of generality, we may assume that there is a finite subset
$\mathcal{F}_0$ in the unit ball of $A$ such that
\[
\mathcal{F} = \{e_{i, j} \otimes a \,\vert\, 
a \in \mathcal{F}_0,\,\, 1 \leq i, j \leq n \}.
\]

Let $\ep_0 = \frac{\ep}{9}$. 
Let $\mathcal{G}_0 = \{e_{i, j} \otimes 1_A\}_{1 \leq i, j \leq n}$. 
Since $M_n$ is semiprojective, 
there is a $\delta_0 > 0$ such that, 
for any $C^*$-algebra $B$,
any $\mathcal{G}_0$-$\delta_0$-multiplicative c.p.c. map 
$L \colon M_n(A) \rightarrow B$ and 
any projection $p\in B$ with $\|p - L(e_{1,1})\| < \delta_0$,
there are matrix units $\{w_{i, j}\}$ in $B$
such that $\| w_{i, j} - L(e_{i, j})\| < \ep_0$ and $w_{1,1} = p$.

Since $A$ is weakly stable with respect to $\mathbf{D}$,
we can choose $\delta_1 > 0$ and a finite subset $\mathcal{G}_1 \in A$
satisfying the following:
for any $B \in \mathbf{D}$,
any $\mathcal{G}_1$-$\delta_1$-multiplicative c.p.c. map $L \colon A \rightarrow B$,
there is a homomorphism $\phi \colon A \rightarrow B$ such that
\[
\| \phi(a) - L(a) \| < \ep_0, \quad \text{for all}\,\, a \in \mathcal{F}_0.
\]

Now let $\delta = \min \{\delta_0, \delta_1, \frac{\ep}{3}\}$, and 
\[
\mathcal{G} = \mathcal{F} \cup \mathcal{G}_0 \cup 
\{ e_{i, j} \otimes a \,\vert\, a \in \mathcal{G}_1, 1 \leq i, j \leq n \}.
\]
Suppose that $B \in \mathbf{D}$ 
and $L \colon M_n(A) \rightarrow B$ is a $\mathcal{G}$-$\delta$-multiplicative c.p.c. map. 
Let $L_1 \colon A \rightarrow B$ be defined by 
\[
L_1(a) =L(e_{1, 1} \otimes a), \quad \text{for all}\,\, a \in A.
\]
It is clear that $L_1$ is a $\mathcal{G}_1$-$\delta_1$-multiplicative c.p.c. map.
Therefore there is a homomorphism
$\phi_1 \colon A \rightarrow B$ such that
\[
\| \phi_1(a) - L_1(a) \| < \ep_0, \quad \text{for all}\, a \in \mathcal{F}_0.
\]
Since $\delta < \delta_0$, $\mathcal{G} \supset \mathcal{G}_0$, 
there are matrix units  $\{w_{i, j}\}$ in $B$
such that $\| w_{i, j} - L(e_{i, j})\| < \ep_0$ and $w_{1,1} = \phi_1(1_A)$.
Now define a homomorphism 
\[
\phi \colon M_n(A) \rightarrow B, \quad \phi(e_{i, j} \otimes a) = w_{i, 1}\phi_1(a)w_{1, j},\quad \text{for all\,\,} a \in A {\rm \,\, and\,\,}
1\leq i, j \leq n.
\]

We can compute that, for any $b = e_{i, j} \otimes a$ in $\mathcal{F}$,
\begin{align*}
\| \phi(b) - L(b) \| 
& < \| w_{i, 1}\phi_1(a)w_{1, j} - L(e_{i, 1} \otimes 1_A)L(e_{1, 1} \otimes a)L(e_{1, j} \otimes 1_A) \| + 2 \delta \\
& \leq \| w_{i, 1} - L(e_{i, 1} \otimes 1_A)\| + \|\phi_1(a) - L_1(a)\|
+ \| w_{1, j} - L(e_{1, j} \otimes 1_A) \|+ 2 \delta\\
&< 3\ep_0 + 2 \delta \leq \ep.
\end{align*}
\end{proof}

\begin{proposition}\label{P_wsHA}
Let $C = PM_n(C(X))P$ be a homogeneous algebra. Let $\mathbf{D}$ be a class of unital $C^*$-algebras
which is closed under taking matrix algebras and unital hereditary subalgebras.
Suppose that $C(X)$ is weakly stable with respect to $\mathbf{D}$. Then $C$ is weakly stable with respect to $\mathbf{D}$.
\end{proposition}

\begin{proof}
We first claim that we only need to prove the unital case, i.e., an almost unital almost multiplicative map
is close to a unital homomorphism. Indeed, if the unital case is true, and suppose
$L \colon C \rightarrow B$ is a $\mathcal{G}$-$\delta$-multiplicative c.p.c. map for some $B \in \mathbf{D}$.
Enlarge $\mathcal{G}$ if necessary, we may assume that $1_C \in \mathcal{G}$. Therefore 
$L(1_C)$ is almost a projection $p \in B$. Then the map $\tilde{L} \colon C \rightarrow pBp$
defined by $\tilde{L} (a) = pL(a)p$, for $a \in C$ is almost unital and almost multiplicative.
Therefore it is close to a homomorphism, so does $L$.

We now prove the unital case. 
We note that there is some positive integer $m$ and a projection $e \in M_m(C)$ such that $eM_m(C)e \cong C(X)$.
Also there is another positive integer $l$, a projection $Q \in M_l(eM_m(C)e) \subset M_{lm}(C)$
and a unitary $W \in M_{lm}(C)$ such that $WQW^*=P \otimes e_{1, 1} \in M_{lm}(C)$.

Let $\ep > 0$ and finite subset $\mathcal{F} \subset C$ be given. 
Without loss of generality, we may assume that $P \in \mathcal{F}$ and every element in
$\mathcal{F}$ has norm less or equal to 1.

Let $\ep_0 = \frac{\ep}{2}$. There is a $\delta_0 < \ep_0$ such that, whenever $p, q$ are two projections in
a unital $C^*$-algebra $A$ such that $\| p - q\| < \delta_0$, then there is a unitary $u \in A$
such that $\| u - 1\| < \ep_0$ and $upu^* = q$.

Let $\mathcal{F}_0 = W^*(\mathcal{F} \otimes e_{1, 1}) W$. Then $\mathcal{F}_0 \in M_l(eM_m(C)e)$.
Since $M_l(eM_m(C)e) \cong M_l(C(X))$, it is weakly stable with respect to $\mathbf{D}$ by Lemma \ref{L_wstable_matrix}. Therefore there is a finite subset $\mathcal{G}_1 \subset M_l(eM_m(C)e)$ and some $\delta_1$ with $0 < \delta_1 < \frac{\delta_0}{5}$ such that,
whenever $B$ is a $C^*$-algebra in $\mathbf{D}$ and $L \colon M_l(eM_m(C)e) \rightarrow B$ is
a $\mathcal{G}_1$-$\delta_1$-multiplicative c.p.c. map, there is a homomorphism 
$\phi \colon A \rightarrow B$ such that 
\[
\| \phi(a) - L(a) \| < \frac{\delta_0}{5}, \text{\,\,for all\,\,}a \in \mathcal{F}_0.
\]

Let $\mathcal{G}_2 = \{W, W^*\}$. There is a $\delta_2$ such that, for any unital $C^*$-algebra $B$
and any $L \colon M_{lm}(C) \rightarrow B$, if $L$ is a $\mathcal{G}_2$-$\delta_2$-multiplicative
map such that $\| L(1_{M_{lm}(C)}) -1_B \| < \delta_2$, then there is a unitary $V \in B$ such that
$\| L(W) - V \| < \delta_1$. 

Let 
\[
\mathcal{G}_3 = \mathcal{G}_1 \cup \mathcal{G}_2 \cup \{W^*(a \otimes e_{1, 1})W\,\vert\, a \in \mathcal{F}\}\,\, {\rm and}\,\, \delta_3 = \min\{\delta_1, \delta_2\}.
\]
It is not difficult to see that, there exists a finite subset $\mathcal{G}$ in $C$ and
$\delta$ with $0 < \delta < \delta_3$ such that, whenever $L \colon C \rightarrow B$ is $\mathcal{G}$-$\delta$-multiplicative,
then $L \otimes \mathrm{id}_{lm} \colon M_{lm}(C) \rightarrow B \otimes M_{lm}$ is
$\mathcal{G}_3$-$\delta_3$-multiplicative. 

Let $B$ be a $C^*$-algebra in $\mathbf{D}$. Suppose that $L \colon C \rightarrow B$ is
a $\mathcal{G}$-$\delta$-multiplicative c.p.c. map such that $\|L(P) - 1_B\| < \delta$.

Let $L_1 = L \otimes \mathrm{id}_{M_{lm}} \colon M_{lm}(C) \rightarrow B \otimes M_{lm}$.
Then there is a unitary $V \in M_{lm}(B)$ such that $\| L_1(W) - V \| < \delta_1$. Also there is a homomorphism
$\phi \colon M_l(eM_m(C)e) \rightarrow M_{lm}(B)$ such that
\[
\| \phi(a) - L_1(a)\| < \frac{\delta_0}{5}, \quad \text{for all\,\,} a \in \mathcal{F}_0.
\]

We can compute that
\[
\| V\phi(Q)V^* - p \otimes e_{1,1} \| < \| L_1(W)L_1(W^*(P \otimes e_{1, 1})W)L_1(W) - L_1(P\otimes e_{1,1})\|
+ \frac{\delta_0}{5} + 2 \delta_1 < \frac{\delta_0}{5} +  4\delta_1 \leq \delta_0.
\]

Therefore there is a unitary $U \in M_{lm}(B)$ such that $UV\phi(Q)V^*U^* = p \otimes e_{1, 1}$ and
$\|U - 1_{M_{lm}(B)}\| < \ep_0$.

Define a homomorphism
\[
\tilde{\phi} \colon C \rightarrow B \otimes M_{lm}, \quad \tilde{\phi}(a) = UV\phi(W^* (a \otimes e_{1, 1}) W)V^*U^*.
\]
Since $\tilde{\phi}(1_C) = p \otimes e_{1, 1} \in B \otimes e_{1, 1}$, we may regard $\tilde{\phi}$
as a homomorphism from $C$ to $B$. For any $a \in \mathcal{F}$, we have:
\begin{align*}
\| \tilde{\phi}(a) - L(a) \| & = \| UV\phi(W^* (a \otimes e_{1, 1}) W)V^*U^* - L_1(a \otimes e_{1, 1})\|\\
&< \| V\phi(W^* (a \otimes e_{1, 1}) W)V^* - L_1(W)L_1( W^*(a \otimes e_{1, 1})W)L_1(W)^*\| + 2\ep_0 + 2 \delta_1\\
&< \| \phi(W^* (a \otimes e_{1, 1}) W) - L_1( W^*(a \otimes e_{1, 1})W)\| + 2\ep_0 + 4 \delta_1\\
&< \ep_0 + 4 \delta_1 + \frac{\delta_0}{5} < \ep_0 + \delta_0 < \ep.
\end{align*}
\end{proof}

As a immediate corollary, we have the following stability result:
\begin{theorem}
Let $\Theta = (\theta_{j, k})_{n\times n}$ be a  real skew-symmetric matrix. If all entries $\theta_{j, k}$
are rational, then the rotation relation with respect to $\Theta$ is stable in the class of unital simple
purely infinite $C^*$-algebras.
\end{theorem}

\begin{proof}
As we discussed at the beginning of section \ref{S_SofTR1}, stability of rotation relation with respect to $\Theta$
is equivalent to that almost homomorphism from $\mathcal{A}_{\Theta}$ is close to a homomorphism.
If all entries of $\Theta$ are rational, then $\mathcal{A}_{\Theta}$ is Morita equivalent to $C(\T^n)$,
by Corollary 4.2 of \cite{Elliott-Lihanfeng}. By a similar argument as in Proposition \ref{P_RAisAH},
we have $\mathcal{A}_{\Theta} \cong PM_m(C(\T^n))P$, for some positive integer $m$ and some
projection $P \in M_m(C(\T^n))$. Let $\mathbf{D}$ be the class of unital simple purely infinite $C^*$-algebras. By Theorem 1.19 of \cite{Lin97}, $C(\T^n)$ is weakly stable with respect to $\mathbf{D}$.
By Proposition \ref{P_wsHA}, $\mathcal{A}_{\Theta}$ is also weakly stable with respect to $\mathbf{D}$.
\end{proof}

The following is an immediate consequence of Theorem 7.6 of \cite{LinTAM04}:
\begin{theorem}\label{T_stabilityPI}
Let $\Theta = (\theta_{j, k})_{n\times n}$ be a  real skew-symmetric matrix. If $\Theta$ is non-degenerate, then the rotation relation with respect to $\Theta$ is stable in the class of unital simple
nuclear purely infinite $C^*$-algebras.
\end{theorem}

In the general case, we can prove stability result with an additional injective condition. 
We remark here that, unlike the stably finite case, we don't know any counterexamples 
in the purely infinite case even without any injective condition. So it is possible that
this injective condition can be removed.

\begin{theorem}
Let $\Theta = (\theta_{j, k})_{n \times n}$ be a real skew-symmetric matrix.
Let $u_1, u_2, \dots, u_n$ be the canonical unitaries in $\mathcal{A}_\Theta$. 
Then for any $\varepsilon>0$, there exists $\delta > 0$,
a finite subset $S \subset \C$,
and a positive integer $N $ satisfying the following:
For any unital simple nuclear purely infinite $C^*$-algebra $A$ 
and any unitaries $v_1,v_2, \dots, v_n\in A$ such that
\begin{enumerate}
\item $\|v_kv_j-e^{2\pi i \theta_{j,k}}v_jv_k\|<\delta, \,\, 1 \leq j , k \leq n$;
\item $\|\sum_{-N \leq l_1, l_2, \dots, l_n \leq N}\lambda_{l_1,l_2,\dots,l_n}v_1^{l_1}v_2^{l_2}\cdots v_n^{l_n}\|
> \frac{2}{3}\|\sum_{-N \leq l_1, l_2, \dots, l_n \leq N}\lambda_{l_1,l_2,\dots,l_n}u_1^{l_1}u_2^{l_2}\cdots u_n^{l_n}\|$,
where\\
 $\lambda_{l_1,l_2,\dots,l_n}$ ranges over $S$,
\end{enumerate}
\
 there exist unitaries $\tilde{v}_1,\tilde{v}_2, \dots, \tilde{v}_n\in A$ such that
$$\tilde{v}_k\tilde{v}_j=e^{2\pi i\theta_{j,k} }\tilde{v}_j\tilde{v}_k\,\,\,\,\mbox{and}\,\,\,\,\|\tilde{v}_j-v_j\|<\varepsilon,\,\,\,j,k=1,2, \dots, n.$$
\end{theorem}

\begin{proof}
Let $\mathcal{F} = \{u_1, u_2, \dots, u_n\}$. Let $\ep > 0$ be given.
By Theorem 7.5 of \cite{LinTAM04}, there is a finite subset $\mathcal{G}_0 \subset \mathcal{A}_{\Theta}$
and a $\delta_0 > 0$ such that, whenever $A$ is a unital simple nuclear purely infinite $C^*$-algebra and
$L \colon \mathcal{A}_{\Theta} \rightarrow A$ is a $\mathcal{G}_0$-$\delta_0$-multiplicative c.p.c. map
such that $\| L(a) \| > \frac{1}{2} \| a \|$, for all $a \in \mathcal{G}_0$, then there is a homomorphism
$\phi \colon \mathcal{A}_{\Theta} \rightarrow A$ such that
\[
\| \phi(b) - L(b) \| < \frac{\ep}{2}, \quad \text{for all\,\,} b \in \mathcal{F}.
\]

Without loss of generality, we assume that $\| a \| \leq 1$, for all $a \in \mathcal{G}_0$.
Let $\sigma = \min \{ \| a \| \,\vert\, a \in \mathcal{G}_0 \}$. Set $\ep_0 = \frac{\sigma}{24}$.

Since $\mathcal{A}_{\Theta}$ is generated by $u_1, u_2, \dots, u_n$, there is a 
positive integer $N > 0$ and a finite subset $S \subset \C$ such that, for any $a \in \mathcal{G}_0$,
there is some $b$ of the form
\[
b = \sum_{-N \leq l_1, l_2, \dots, l_n \leq N}\lambda_{l_1,l_2,\dots,l_n}u_1^{l_1}u_2^{l_2}\cdots u_n^{l_n}, \quad\lambda_{l_1,l_2,\dots,l_n} \in S
\]
with $\|a - b\| < \ep_0$.

Let $M=\max\{|\lambda||\lambda\in S\}.$ 
Let $\ep_1 = \min \{\frac{\ep}{2}, \frac{\ep_0}{MN^N}\}$. Set
\[
\eta = \frac{\sigma}{48MN^{N+1}}, \quad
\mathcal{G} = \{u_1^{l_1}u_2^{l_2}\cdots u_n^{l_n} \,\vert\, -N \leq l_1, l_2, \dots, l_n \leq N\}.
\]
Choose $\delta$ according to
$\mathcal{G}$, $\eta$ and $\ep_1$ (in place of $\ep_0$) as in Proposition \ref{P_AlmostHExist}.

Now let $A$ be a unital simple nuclear purely infinite $C^*$-algebra,
let $v_1,v_2, \dots, v_n\in A$ be unitaries such that
\begin{enumerate}
\item $\|v_kv_j-e^{2\pi i \theta_{j,k}}v_jv_k\|<\delta,\quad 1\leq j , k \leq n$;
\item $\|\sum_{-N \leq l_1, l_2, \dots, l_n \leq N}\lambda_{l_1,l_2,\dots,l_n}v_1^{l_1}v_2^{l_2}\cdots v_n^{l_n}\|
> \frac{2}{3}\|\sum_{-N \leq l \leq N}\lambda_{l_1,l_2,\dots,l_n}u_1^{l_1}u_2^{l_2}\cdots u_n^{l_n}\|$,
where  $\lambda_{l_1,l_2,\dots,l_n}$ ranges over $S$.
\end{enumerate}

By Proposition \ref{P_AlmostHExist}, there is a $\mathcal{G}$-$\eta$-multiplicative c.p.c.
map $L \colon \mathcal{A}_{\Theta} \rightarrow A$ such that $\| L(u_i) - v_i \| < \frac{\ep_0}{MN^N}$.
Now for any $a \in \mathcal{G}_0$, find some $b$ of the form
\[
b = \sum_{-N \leq l_1, l_2, \dots, l_n \leq N}\lambda_{l_1,l_2,\dots,l_n}u_1^{l_1}u_2^{l_2}\cdots u_n^{l_n}, \,\,\,\lambda_{l_1,l_2,\dots,l_n} \in S
\]
with $\|a - b\| < \ep_0$. We can then compute that
\begin{align*}
\| L(a) \| & > \| L(b) \| - \ep_0\\
&\geq \| \sum_{-N \leq l_1, l_2, \dots, l_n \leq N}\lambda_{l_1,l_2,\dots,l_n}L(u_1)^{l_1}L(u_2)^{l_2}\cdots L(u_n)^{l_n}\| - M(2N)N^N\eta - \ep_0\\
&\geq \| \sum_{-N \leq l_1, l_2, \dots, l_n \leq N}\lambda_{l_1,l_2,\dots,l_n}v_1^{l_1} v_2^{l_2} \cdots v_n^{l_n}\| - MN^N\ep_1- M(2N)N^N\eta - \ep_0\\
& > \frac{2}{3} \| b \| - \ep_0- M(2N)N^N\eta - \ep_0 \\
& > \frac{2}{3} \| a \| - \frac{2}{3}\ep_0 - \ep_0- \frac{\sigma}{24} - \ep_0\\
&\geq \frac{2}{3} \| a \| - \frac{\sigma}{6} > \frac{1}{2} \| a \|.
\end{align*}

By Theorem 7.5 of \cite{LinTAM04}, 
there is a homomorphism $\phi \colon \mathcal{A}_{\Theta} \rightarrow A$ such that
\[
\| \phi(u_j) - v_j \| < \|\phi(u_j) -L(u_j)\| + \|L(u_j) - v_j\| < \frac{\ep}{2} + \frac{\ep}{2} = \ep.
\]
Let $\tilde{v}_j = \phi(u_j)$.  These are desired unitaries.
\end{proof}

\section*{Acknowledgments}
The first named author was supported by the National Natural Science
Foundation of China (No. 11401256) and the Zhejiang Provincial Natural Science Foundation of China (No. LQ13A010016). 

Email: huajiajie2006@hotmail.com

Jiajie Hua

College of Mathematics, Physics and Information engineering
 
 Jiaxing University
 
Jiaxing,  Zhejiang, 314000, China

\

Email: qingyunw@uoregon.edu 

Qingyun Wang

Department of Mathematics

University of Oregon

Euegne, 97403, USA

\end{document}